\documentclass[10pt]{article}
\usepackage[english]{babel}
\usepackage{amsmath, amsthm}
\usepackage{fancyhdr}

\newcommand{\version}{\today}

\usepackage{amsmath, amsthm}
\usepackage{amscd}
\usepackage{amsfonts}
\usepackage{amssymb}
\usepackage{mathrsfs}
\usepackage{eucal}
\usepackage{url}

\newcommand{\R}{\mathbb{R}}
\newcommand{\C}{\mathbb{C}}
\DeclareMathOperator{\T}{\mathcal{T}}
\DeclareMathOperator{\Lie}{Lie}
\DeclareMathOperator{\rU}{U}
\newcommand{\Z}{\mathbb{Z}}

\DeclareMathOperator{\Cx}{Cx}
\DeclareMathOperator{\Univ}{\mathcal{U}}
\newcommand{\cO}{\mathcal{O}}
\DeclareMathOperator{\Hom}{Hom}
\DeclareMathOperator{\KS}{KS}
\DeclareMathOperator{\Gr}{Gr}
\DeclareMathOperator{\D}{Diff}

\DeclareMathOperator{\Id}{Id}
\DeclareMathOperator{\Hor}{Hor}
\newcommand{\ox}{\otimes}
\newcommand{\x}{\times}
\newcommand{\1}{\sqrt{-1}}
\newcommand{\eps}{\varepsilon}

\newcommand{\pdb}{\bar{\pd}}

\numberwithin{equation}{section}

\title{Haupt--Kapovich theorem revisited}
\author{rodion~n.~d\'eev}

\newcommand{\arr}{\xrightarrow}
\newcommand{\pd}{\partial}

\newtheorem{pr}{Proposition}[section]
\newtheorem{corol}[pr]{Corollary}%[section]
\newtheorem*{problem}{Problem}
\newtheorem*{thrm}{Theorem}
\newtheorem*{fact}{Fact}
\newtheorem*{lm}{Lemma}

\theoremstyle{definition}
\newtheorem{defn}{Definition}
\newtheorem*{expl}{Example}

\renewcommand{\P}{\mathrm{P}}
\DeclareMathOperator{\Gris}{Gris}
\DeclareMathOperator{\ord}{ord}
\DeclareMathOperator{\MCG}{MCG}
\DeclareMathOperator{\Sp}{Sp}
\DeclareMathOperator{\SL}{SL}
\DeclareMathOperator{\Res}{Res}
\DeclareMathOperator{\Grsymp}{Grsymp}

\begin{document}
\maketitle

{\it \rightline{to Misha, Ilya, Vitali, and Katia Kapovich}}

\vspace{6mm}

\begin{abstract}
A theorem of O.~Haupt, rediscovered by M.~Kapovich and celebrated by his proof invoking Ratner theory, describes the set of de Rham cohomology classes on a topological orientable surface, which can be realized by an abelian differential in some respective complex structure, in purely topological terms. We give a simplification of Kapovich's proof and make an attempt to describe similarly pairs and triples of cohomology classes, which can be realized by abelian differentials in some complex structure. This leads to some interesting problems in algebraic geometry of curves, and gives an unexpected local description of the Teichm\"uller space.
\end{abstract}
\tableofcontents

\section{Introduction}

The problem which the Haupt--Kapovich theorem concerns is the following. Let $S$ be a topological genus $g>1$ orientable surface. Its first cohomology $H^1(S,\C)$ is the complexification of a lattice endowed with an integral skew-symmetric form, the intersection pairing. If the complex structure $I$ on the surface is fixed, making it into a complex curve, one can consider the space $H^{1,0}(S,I)$ of cohomology classes of holomorphic 1-forms. Whenever another complex structure is chosen, another $H^{1,0}$ subspace in cohomology arises, so it of course cannot be reconstructed from the linear-algebraic data of topological origin on $H^1(S,\C)$. But can the union of all the $H^{1,0}$ subspaces for all possible complex structures be captured by topology? The answer is positive, and is given by the following

\begin{thrm}[O.~Haupt, 1920 \cite{H}, M.~Kapovich, 2000 \cite{K}]The union $$\bigcup_{I\in\T_S}H^{1,0}(S,I) \subset H^1(S,\C)$$ of all the $H^{1,0}$ subspace over all possible complex structures is precisely the subset of the classes $[\alpha]$ which satisfy the following conditions:
\begin{enumerate}
\item $\1[\alpha]\wedge[\bar{\alpha}] \geqslant 0,$
\item Whenever the set of periods $\left\{\int_\gamma\alpha\right\}_{\gamma \in H_1(S,\Z)}$ is a lattice in $\C$, one has $\1[\alpha]\wedge[\bar{\alpha}]$ greater than the covolume of this lattice.
\end{enumerate}
\end{thrm}

Here $\T_S$ stands for the Teichm\"uller space of complex structures on $S$ modulo isotopy. The necessity of these conditions is obvious: the first follows from the fact that $\1dz\wedge d\bar{z}$ is a positive volume form on a Riemannian surface; as for the second, if $\alpha$ is a holomorphic representative, the multivalued integral $\int_{z_0}^z\alpha$ descends to a holomorphic mapping onto an elliptic curve, the quotient of $\C$ by the periods of $\alpha$, and the degree of a map from a genus $g>1$ curve onto an elliptic curve is at least two, so its area w.~r.~t. the class $[\alpha]$ is at least twice as the area of the elliptic curve, i.~e. the covolume of the period lattice. However, the sufficiency of these conditions is nontrivial. Although it can be proved by elementary means (and this way has been proven by Haupt), a more elegant proof by Kapovich is based on the ergodic theory, especially the Ratner theorem.

Since the subset of realizable classes is invariant under scaling, one may consider it as a question about characterization of linear subspaces, which fall into $H^{1,0}(S,I)$ for some complex structure $I$. Bogomolov noticed that a variant of this question turns out to be interesting:

\begin{problem}Describe the union $$\bigcup_{I\in\T_S} {\Gr}{\left(k,H^{1,0}(S,I)\right)} \subset {\Gr}{\left(k, H^1(S,\C)\right)}$$ in terms of topology of $S$.
\end{problem}

In this paper, we slightly modify the proof of Kapovich in order to fit better for dealing with this problem.

The paper is organized as follows. In the Section \ref{htht}, we discuss the openness of the set of representable classes (or pairs, or triples). In the subsection \ref{htht-ks} the notation is fixed and familiar facts about the Teichm\"uller space, Gau\ss--Manin connection, Kodaira--Spencer class for algebraic curves are recollected. In the subsection \ref{htht-kap}, we introduce the polyperiod mappings, and restate an~openness theorem due to Hejhal and Thurston, which is of great importance in the Kapovich's paper. In the subsection \ref{htht-weak}, a weaker version of Hejhal--Thurston theorem for pairs and triples of abelian differentials is proved, and the algebraic-geometric condition on the pairs in which this theorem fails, is outlined. In the subsection \ref{htht-counter}, a geometric construction of certain examples of failure of Hejhal--Thurston theorem for pairs and triples is provided.

The Section \ref{ratner} invokes ergodic-theoretic considerations. Moore's theorem implies that the sets of representable classes are dense, which in particular implies the Proposition \ref{three-diff}, which is one of the main results of the paper.

\begin{thrm}[Proposition \ref{three-diff}]
Let $(S, I)$ be a complex curve of genus at least four, and $\tau \subset H^{1,0}(S,I)$ be a generic three-dimensional subspace. There exist a neighborhood $U \subset \Gris(3,2g)$ in the isotropic Gra\ss mannian w.~r.~t. the intersection pairing containing $\tau$ s.~t. for any $\tau' \in U$ there exists a unique complex structure $I' = I(\tau')$ s.~t. one has $\tau' \subset H^{1,0}(S,I')$. In other words, deformation of a generic triple of abelian differentials determines a unique local deformation of a complex structure, so that the deformed cohomology classes would be represented by abelian differentials in the deformed complex structure.
\end{thrm}

Then we remind of the famous $\SL(2, \R)$-action on the moduli space of abelian differentials, and exploit it to give a simplified version of Kapovich's classification of orbits by the mean of Ratner theory. In the Section \ref{el}, the elliptic subspaces are studied, this allows us to finish our version of Kapovich's proof (Proposition \ref{hk-new}). In the end of the section, we formulate an analogue of the Haupt--Kapovich condition for elliptic pairs (Proposition \ref{hk-el-pair}). In the meantime, we characterize possible number of nodes on a curve of given genus and given self-intersection on an abelian surface in a manner which resembles similar classical statement for plane curves due to Severi.

In the Section \ref{chronos}, we describe some geometry of the isoperiodic locus with prospect on so-called lesser isoperiodic foliation defined in the subsection \ref{htht-weak} (Definition \ref{lesser}), deducing from it certain higher-order Legendrian property for the Schottky locus (Proposition \ref{legen}). There we also prove a statement similar to Proposition \ref{three-diff}, but restricted to hyperelliptic curves:

\begin{thrm}[Proposition \ref{hype-two-diff}]
Let $(S, I)$ be a hyperelliptic complex curve of genus at least three, and $\tau \subset H^{1,0}(S,I)$ be a generic two-dimensional subspace. There exists a neighborhood $U \subset \Gris(2,2g)$ containing $\tau$ s.~t. for any $\tau' \in U$ there exists a unique hyperelliptic complex structure $I' = I(\tau')$ s.~t. one has $\tau' \subset H^{1,0}(S,I')$. In other words, deformation of a generic pair of abelian differentials determines a unique local deformation of a hyperelliptic complex structure, so that the deformed cohomology classes would be represented by abelian differentials in the deformed complex structure.
\end{thrm}

The following theorem, which seems to be new, follows from this construction:

\begin{thrm}[Corollary \ref{def-to-hype}]Let $\{C \subset A^2\}$ be a generic pair of an abelian surface $A^2$ and a curve $C$ on it with no worse than normal crossings. Then there exists a hyperelliptic curve on $A^2$ in the same homology class $[C]$.
\end{thrm}

This theorem is a sibling of Bogomolov--Mumford theorem on existence of rational curves on any K3 surface. After all, in the subsection \ref{recip}, we prove a statement about values of two holomorphic differentials at zeroes of a third, which is similar to the reciprocity law on curves for meromorphic functions, yet probably cannot be reduced to it.

%Sections \ref{common-zeroes} and \ref{lace} are devoted to refinement of these results, in order to eliminate the condition of common zeroes being disjoint (which is only possible up to certain extent).

In what follows, we use italics when reminding the well-known definitions, and bold when introducing our own.

\section{Hejhal--Thurston `holonomy' theorem}
\label{htht}

\subsection{Three incarnations of Kodaira--Spencer tensor}
\label{htht-ks}
This subsection is a reminder of what is widely known. For an introduction, see an excellent book \cite{TrAG}.

\paragraph{Shape operator of the Hodge bundle.} Consider the space $\Cx(S)$ of tensors $I \in C^\infty(T^*S \ox TS)$ satisfying $I^2=-\Id_{TS}$ with the usual topology on smooth sections of a vector bundle. Any such tensor gives rise to a complex structure since $\dim S = 2$. The diffeomorphism group $\D(S)$ acts on this space via pullbacks.

\begin{defn}
The quotient $\T(S) = \Cx(S)/\D^0(S)$ by the connected component of the diffeomorphism group is called the {\it Teichm\"uller space}.
\end{defn}

\begin{fact}
The Teichm\"uller space of a genus $g$ surface is isomorphic, as a quotient of a Fr\'echet manifold by a Fr\'echet--Lie group, to a finite-dimensional smooth manifold of dimension $6g-6$ (actually, an open ball). It carries a natural complex structure and a holomorphic fibration called the {\normalfont universal curve} $\Univ_S\arr{\pi}\T_S$ with fiber over a point $I$ biholomorphic to $(S,I)$.
\end{fact}

Let $H^1_{\Z}\T_S \to \T_S$ be the sheaf ${R^1\pi_*}{\left(\underline{\Z}_{\Univ_S}\right)}$. It is actually a sheaf of sections of the local system $H^1_{\Z}\T_S$ with stalk over the point $I$ being the integral first cohomology of the fiber $\pi^{-1}(I)$. It can be complexified to obtain the vector bundle $H^1\T_S = H^1_{\Z}\T_S\ox\cO_{\T_S}$ with fiber over the point $I$ being the vector space $H^1(\pi^{-1}(I),\C) \approx H^1(S,\C)$.

\begin{defn}
The only connection $\nabla^{GM}$ on the bundle $H^1\T_S$ making the local sections of $H^1_{\Z}\T_S \subset H^1\T_S$ parallel is called the {\it Gau\ss--Manin connection}.
\end{defn}

Since in each neighborhood it possesses a basis of parallel sections, this connection is flat, and since the Teichm\"uller space is simply connected, it gives a trivialization of the bundle, i.~e. the projection $H^1\T_S \to V := H^1(S,\C)$ which we shall denote by $\varkappa$. It sends a holomorphic 1-form $\alpha$ to the linear function $\gamma\mapsto\int_\gamma\alpha$ on homology, which associates to every cycle the period of $\alpha$ along it, and hence is sometimes referred to as the {\it period map}.

\begin{defn}
The cohomology bundle $H^1\T_S$ has a subbundle $\Omega\T_S \subset H^1\T_S$ s.~t. $\Omega\T_S|_I = H^{1,0}(S,I) \subset H^1(S,I) = H^1\T_S(I)$, i.~e.~the bundle of abelian differentials. It is called the {\it Hodge (sub)bundle} and is not parallel w.~r.~t.~the Gau\ss--Manin connection. It can be also defined abstractly as ${\pi_*}{\left(\Omega^1_{\Univ_S/\T_S}\right)}$.
\end{defn}

Since no abelian differential has real coefficients, the Hodge subbundle does not intersect its complex conjugate. Hence one can consider the projection $\varpi \colon H^1\T_S \to \overline{\Omega\T_S}$ with kernel $\Omega\T_S$. 

\begin{defn}
The {\it Kodaira--Spencer tensor} $\KS \colon T\T \ox \Omega\T_S \to \overline{\Omega\T_S} \cong \left(\Omega\T_S\right)^*$ is the shape operator of the Gau\ss--Manin connection w.~r.~t.~this splitting: $$\KS_v(\alpha) = \varpi\left(\nabla^{GM}_v\alpha\right).$$
\end{defn}

It is a standard fact that this operator is linear over $C^\infty\left(\T_S\right)$ in both indices.

\paragraph{First cohomology space of the tangent bundle.} The tangent space $T_I\Cx(S)$ is the space of all tensors $A \colon TS \to TS$ s.~t. the operator $I+\eps A$ for $\eps^2=0$ squares to $-\Id_{TS}$. Expanding, one gets $I^2 + \eps(AI+IA)+\eps^2A^2 = -\Id = I^2$, or just $AI+IA = 0$. The space of such tensors is acted upon by the Lie algebra of the diffeomorphism group, i.~e. the Lie algebra of vector fields.

\begin{fact}
The quotient $$\frac{\{A \colon TS\to TS\mid AI+IA = 0\}}{\{\Lie_vI \mid v\in C^\infty(TS)\}}$$ is indeed isomorphic to the tangent space of the Teichm\"uller space at $I$.
\end{fact}

This quotient resembles the first de Rham cohomology: indeed, it is the quotient of 1-forms (with vector fields coefficients) subject to some closedness condition by 1-forms coming from vector-valued 0-forms (i.~e. vector fields) via derivation.

Let $E\to X$ be a complex vector bundle over a complex manifold. To give a complex structure on its total space in which the fibers are complex submanifolds and the projection map is holomorphic is the same as to give the Dolbeault operator $$\pdb \colon \Gamma(E) \to \Omega^{0,1}(X,E).$$ Its extension to all forms with coefficients in sections of $E$ via the usual Leibniz rule satisfies $\pdb^2=0$, hence one can speak of cohomology with coefficients in a holomorphic vector bundle. In particular, $H_{\pdb}^0(E)$ is just the space of holomorphic sections of $E$.

\begin{fact}
The tangent space to the Teichm\"uller space at point $I$ can be canonically identified with the cohomology space $H^1(S,TS)$ for the standard holomorphic structure on the tangent bundle of the complex curve $(S,I)$.
\end{fact}

More geometric way to see the correspondence between the first-order deformations of the complex structure and the first cohomology of the holomorphic tangent bundle is as follows. Consider a first-order deformation $\mathfrak{X} \to \Delta = \mathrm{Spec}~\C[h]/(h^2)$ with central fiber $\mathfrak{X}_0 \approx X$, and the short exact sequence of vector bundles $TX \to T\mathfrak{X}|_{X} \to \cO_X$. The isomorphism classes of extensions $E \to E' \to \cO$ are precisely the first cohomology classes from $H^1(X, E)$.

\begin{fact}
The Kodaira--Spencer tensor defined above is the same as the mapping $$\KS \colon H^{1,0}(S,I) \x H^1(S,TS) \to H^{0,1}(S,I)$$ which can be written down on the representatives by the rule $$\KS(\alpha\x v) = \alpha(v(x)),$$ where $\alpha\in\Omega^{1,0}(S,I)$ is a closed $(1,0)$-form, $v \in \Omega^1(S,TS)$ is a vector-valued 1-form vanishing on $(1,0)$-vectors, and the result lies in the space $\Omega^{0,1}_{\mathrm{cl}}(S,I)$ of closed $(0,1)$-forms.
\end{fact}

\paragraph{Multiplication of holomorphic 1-forms.} Let $X$ be an $n$-dimensional complex projective manifold, $K_X$ be its canonical bundle, i.~e. the top exterior power of the holomorphic cotangent bundle, and $E$ a holomorphic vector bundle.

\begin{fact}[Serre duality]$H^i_{\pdb}(E) \cong H^{n-i}_{\pdb}(E^*\ox K_X)^*$.
\end{fact} 

Let $\Omega^p = \Omega^{p,0}$ be the complex exterior $p$-th power of the holomorphic cotangent bundle. The cohomology of these bundle has the following interpretation:

\begin{fact}[Dolbeault theorem]$H^q(\Omega^p) \cong H^{p,q}(X,I)$.
\end{fact}

Applying these isomorphisms, the Kodaira--Spencer map for complex curves can be rewritten as a map $$\KS \colon H^0(K_S) \x H^0(T_S^*\ox K_S)^* \to H^1(\cO_S),$$ or applying the Serre duality again on the right and by adjunction $$\KS \colon H^0(K_S^2)^* \to H^0(K_S)^* \ox H^0(K_S)^*.$$

\begin{fact}
The Kodaira--Spencer map on a curve $\KS \colon H^0(K_S^2)^* \to H^0(K_S)^* \ox H^0(K_S)^*$ is dual to the symmetric multiplication of holomorphic 1-forms, i.~e. the natural map $H^0(K_S)^2 \to H^0(K_S^2)$.
\end{fact}

\subsection{Hejhal--Thurston theorem and polyperiod mappings}
\label{htht-kap}
Let us go back to the differential-geometric setting. So we have the Teichm\"uller space $\T_S$, the vector space $V=H^1(S,\C)$, the vector bundle $H^1\T_S \to \T_S$ with global trivialization $H^1\T_S \to V$, and the subbundle $\Omega\T_S \subset H^1\T_S$. Fix an integer $k$ and consider the bundle of Gra\ss mannian varieties $\Gr(k,\Omega\T_S) \to \T_S$. The composition of inclusion and projection $\Omega\T_S \to H^1\T_S \to V$ induces the map $\Gr(k,\Omega\T_S) \to \Gr(k,V)$ which we shall denote by $\varkappa_k$.

\paragraph{The extreme case: $k=g$.} In this case each fiber of the fibration $\Gr(g,\Omega\T_S)$ is a single point corresponding to the whole $g$-dimensional space $\Omega_I = H^{1,0}(S,I)$, hence $\Gr(g,\Omega\T_S)$ is canonically biholomorphic to $\T_S$. In this case the mapping $\varkappa_g \colon \T_S \to \Gr(g,V)$ is known as the {\it Torelli mapping}.

\begin{fact}[local Torelli theorem for curves]The Torelli mapping $\varkappa_g \colon \T_S \to \Gr(g,2g)$ is locally a holomorphic embedding.
\end{fact}

The image of this embedding is known as the {\it Schottky locus}, and is rather difficult to describe. 

\paragraph{The opposite case: $k=1$.} In this case each fiber of the fibration $\Gr(1,\Omega\T_S) = \P(\Omega\T_S) \to \T_S$ is a projective space. The mapping $\varkappa_1 \colon \P(\Omega\T_S) \to \P(V)$ has been studied by Kapovich (indeed, its image is the projectivization of the cone of representable classes from Haupt--Kapovich theorem). Hence 

\begin{defn}
We shall call the mapping $\varkappa_1$ the {\it period mapping}, and its image the {\bf Kapovich locus}. In the intermediate cases $1 < k < g$ we shall refer to the maps $\varkappa_k$ as to the {\bf polyperiod mappings}, and their images as the {\bf Kapovich--Schottky loci}.
\end{defn}

The following is what Kapovich called the `holonomy theorem', and ascribed it to Hejhal and Thurston.

\begin{thrm}
The period map $\varkappa_1 \colon \P(\Omega\T_S) \to \P(V)$ is open.
\end{thrm}

Kapovich established it by analytical means considering the uniform convergence of developing mappings. We shall prove a stronger statement by the means of tensor calculus.

\begin{pr}
\label{htht-strong}
The differential of the period mapping $\varkappa_1 \colon \P(\Omega\T_S) \to \P(V)$ is everywhere surjective.
\end{pr}
\begin{proof}
The part of Gau\ss--Manin connection preserving the Hodge subbundle splits the tangent bundle $T\P(\Omega\T_S)$ into the vertical subbundle, which is isomorphic to $T_{\langle a\rangle}\P(\Omega_I) = \Hom(\langle a\rangle,\Omega(S,I)/\langle a\rangle)$ for any line $\langle a\rangle \subset \Omega(S,I)$ (and maps by the derivative of the period mapping isomorphically onto the tangent space $T_{\langle a\rangle}\P(\Omega_I) \subset T_{\langle a\rangle}\P(V)$), and the horizontal subbundle $\Hor_{\langle a\rangle, I} \approx T_I\T_S$. By definition, the differential $$d\varkappa_1|_{\Hor_{\langle a\rangle, I}} \colon T_I\T_S \to \Hom\left(\langle a\rangle, H^{0,1}(S,I)\right)$$ is the restriction of the Kodaira--Spencer tensor $\KS \colon T\T_S \x \Omega\T_S \to \overline{\Omega}\T_S$ onto the line spanned by $a$.

Hence the Proposition \ref{htht-strong} is equivalent to the following `nondegeneracy' assertion about the Kodaira--Spencer tensor:

\begin{pr}
For any complex structure $I$ on $S$ and any nonzero class $a \in H^{1,0}(S,I)$ the map given $T_I\T_S \to H^{0,1}(S,I)$ by $$v \mapsto \KS_v(\alpha)$$ is surjective. Equivalently, no matter which nonzero section $\alpha \in H^0(K_S)$ is given, any linear functional from $H^0(K_S)^*$ can be represented by $$\beta \mapsto v(\alpha\ox\beta)$$ for some $v \in H^0(K_S^2)^*$.
\end{pr}
\begin{proof}
The restatement of the proposition is equivalent to the statement that the map $m_\alpha \colon H^0(K_S) \to H^0(K_S^2)$ given by $m_\alpha(\beta) = \alpha\ox\beta$ is injective. This is obvious though, since whenever both $\alpha$ and $\beta$ are nonzero, their product is nonzero away from its $4g-4$ zeroes (counted with multiplicity).
\end{proof}

\begin{defn}
We shall call the range of the map $m_\alpha$ the {\bf dividend subspace} and denote it by $L_\alpha \subset H^0(K_S^2)$. It is precisely the space of holomorphic quadratic differentials divisible by $\alpha$.
\end{defn}
\end{proof}

\subsection{Weak Hejhal--Thurston theorem for pairs and triples}
\label{htht-weak}
In order to generalize the Kapovich's proof to the subspaces in cohomology other than lines, one should determine where the differential of the polyperiod mapping is surjective. The first thing we need to notice is that, in a sense, it never is: the space $V$ carries the skew-symmetric intersection pairing (which we shall denote by $\omega$) given by the wedge product of forms, and two holomorphic 1-forms on a curve wedge multiply to zero. Therefore the Kapovich--Schottky locus lies within the isotropic Gra\ss mannian $\Gris(k, V)$, which is a closed subset of $\Gr(k, V)$. In its turn, in order to conclude whether the polyperiod mapping $\varkappa_k \colon \Gr(k,\Omega\T_S) \to \Gris(k,V)$ is open, one needs to understand first what is the tangent space to the isotropic Gra\ss mannian. Let us start from the standard observation that the symplectic form induces a well-defined pairing $\tau \times V/\tau \to \C$ for any isotropic subspace $\tau \subset V$.

\begin{lm}
Let $\tau \in \Gris(k,V)$ be an isotropic subspace. A vector $v \in T_\tau\Gr(k,V) = \Hom(\tau,V/\tau)$ is tangent to the isotropic Gra\ss mannian $\Gris(k,V) \subset \Gr(k,V)$ iff the corresponding map satisfies $$\omega(x,v(y)) = -\omega(v(x),y)$$ for any $x, y\in \tau$.
\end{lm}
We shall call such maps {\bf balanced}.
\begin{proof}
The infinitesimal displacement of the plane $\tau$ by the mapping $\theta$ consists of the vectors $\{x+\eps \theta(x)\colon x\in\tau\}$ for $\eps^2=0$. This plane is isotropic whenever for any $x,y\in\tau$ one has $0=\omega(x+\eps\theta(x),y+\eps\theta(y)) = \omega(x,y) + \eps(\omega(\theta(x),y) + \omega(y,\theta(x)))$, which is equivalent to the balancedness condition since $\omega(x,y)=0$ for any $x,y\in\tau$.
\end{proof}

Therefore, the polyperiod mapping has surjective differential at point $\tau \subset H^{1,0}(S,I)$ iff any balanced map $\theta \colon \tau \to V/\tau$ can be represented as the Kodaira--Spencer map for some first-order deformation $v$, i.~e. $$\eta \mapsto \KS_v(\eta)|_{\tau} \in \Hom(\tau, H^{0,1}(S,I)) \subset \Hom(\tau, V/\tau).$$

\begin{defn}
If $\tau \subset H^0(K_S)$ is a subspace, we call the kernel of the multiplication map $\tau\ox H^0(K_S) \to H^0(K_S^2)$ the {\bf obscurant subspace}. We call a subspace $\tau$ (or a tuple of forms spanning it) {\bf coprime} if $\dim\tau = 2$ and the obscurant subspace is one-dimensional, or $\dim\tau = 3$ and the obscurant subspace is three-dimensional (note that in this case the multiplication map is also surjective). Otherwise we call $\tau$ {\bf linked}.
\end{defn}

\begin{pr}
\label{krichever}
Let $\tau \subset H^0(K_S)$, and $\dim\tau = 2$ or $3$. Then the derivative of the polyperiod mapping at $\tau \in \Gr(\dim\tau,\Omega\T_S)$ is surjective iff $\tau$ is coprime. Moreover, in the case $\dim\tau = 3$ and $\tau$ being coprime the derivative is bijective.
\end{pr}
\begin{proof}
Let us again consider the space $H^{0,1}(S,I)$, the target of a mapping $\theta \in \Hom(\tau,H^{0,1}(S,I)) \subset \Hom(\tau,V/\tau)$, as the dual of the space of abelian differentials $\Omega(S,I)$. In the case $k=2$, our goal is to show that if $\theta$ is balanced, then it can be realized as the value of the Kodaira--Spencer tensor on some deformation $v \in H^0(K_S^2)^*$. Let $\alpha,\beta\in\tau$ be a basis. We know what the values of $\theta$ on $\alpha$ and $\beta$ are: they should be realized as restrictions of $v$ onto the dividend subspaces $L_\alpha$ and $L_\beta$ (which are identified with the space of abelian differentials by the mappings $m_\alpha$ and $m_\beta$). The balancedness condition means that these functionals on $L_\alpha$ and $L_\beta$ coincide on the line spanned by $\alpha\ox\beta$. Provided this line exhausts the intersection $L_\alpha \cap L_\beta$, one can uniquely extend this pair of functionals to the subspace spanned by $L_\alpha \cup L_\beta$ and then to whole space $H^0(K_S^2)$. Otherwise any $\xi \in L_\alpha \cap L_\beta$ noncollinear to $\alpha\ox\beta$ gives a hyperplane in ${T_{\tau}}{\Gris(2,V)}$, in which the range of the derivative of the polyperiod map is contained. Note however that the intersection $L_\alpha \cap L_\beta$ has the same dimension as the kernel of the map $\tau\ox H^0(K_S) = L_\alpha \oplus L_\beta \to H^0(K_S^2)$, i.~e. the obscurant subspace. Moreover, the obscurant subspace is canonically identified with the subspace of $H^0(K_S^2)$ consisting of quadratic differentials divisible by all the 1-forms in $\tau$ (one can project the kernel onto $L_\alpha$ along $L_\beta$, consider the range of the projection as a subspace in $L_\alpha \subset H^0(K_S^2)$, and note that the composite map does not depend on the choice of $\alpha$ and $\beta$).

In the case $k=3$ and $\tau = \langle \alpha,\beta,\gamma \rangle$ one knows the values of the desired functional $v$ on the subspaces $L_\alpha$, $L_\beta$ and $L_\gamma$. The balancedness condition implies again that these values agree on the lines $\alpha\ox\beta$, $\beta\ox\gamma$ and $\gamma\ox\alpha$. Then it can be extended to a functional on the linear hull of $L_\alpha \cup L_\beta \cup L_\gamma$ no matter what its values were, iff no other relation on monomials would show up after taking symmetric products (i.~e., the obscurant subspace is three-dimensional). The space of relations is precisely the kernel of the natural mapping $L_\alpha \oplus L_\beta \oplus L_\gamma \to H^0(K^2_S)$, and it is three-dimensional iff this map is surjective, i.~e. the union $L_\alpha\cup L_\beta \cup L_\gamma$ spans the whole $H^0(K_S^2)$. In this case the deformation is determined by its values on three forms $\alpha$, $\beta$, $\gamma$ uniquely, so the differential of the polyperiod map is bijective at this point.
\end{proof}

Note that in the case $k=3$ it is not enough to merely ask for the pairwise intersections $L_\alpha \cap L_\beta$, $L_\beta \cap L_\gamma$ and $L_\gamma \cap L_\alpha$ to be one-dimensional, or even to be one-dimensional and not to lie within one plane (much like it happens with the Borromean rings). A counterexample is given by any hyperelliptic curve of genus greater than two, cf.~Proposition~\ref{couplets}.

\begin{corol}
\label{one-to-one}
\begin{enumerate}
\item Let $(S,I)$ be a curve of genus at least two and $\tau \subset \Omega(S,I)$ a plane spanned by a pair of abelian differentials. Then it has a local $(g-2)$-dimensional family of deformations which preserve the periods of the differentials from $\tau$, and when $\tau$ is coprime, the dimension of the deformation space equals exactly $g-2$.
\item Let $(S,I)$ be a curve of genus at least three and $\tau \subset \Omega(S,I)$ a three-dimensional coprime subspace. Then any local deformation preserving the periods of the differentials from $\tau$ is trivial.
\end{enumerate}
\end{corol}

\begin{defn}\label{lesser}
Analogously to the case $k=1$, in which the fibers of the period map $\P(\Omega\T_S) \to \P(V)$ are known as {\it isoperiodic foliation}, we shall refer to the fibers of the polyperiod map $\Gr(2,\Omega\T_S) \to \Gris(2,V)$ as to the {\bf lesser isoperiodic foliation}, since its leaves, after being projected to the Teichm\"uller space, lie within the projections of the leaves of the usual isoperiodic foliation.
\end{defn}

A similar foliation in a slightly more Teichm\"uller-theoretic context, for a pair of meroporphic differentials with real periods, appeared in an early version of a paper of Grushevsky and Krichever \cite{GK} under the name {\it small foliation}. Unlike the lesser isoperiodic foliation, which we know to have singularities, their foliation is conjectured to be smooth.

\begin{pr}
\label{com-zer}
Let $\alpha$ and $\beta$ be two linked 1-forms on a curve $S$ (i.~e. their classes span a linked plane in cohomology). Then they have at least two zeroes in common. If they have exactly two common zeroes, $S$ must be hyperelliptic. In general, if they have exactly $n$ common zeroes, then the gonality of $S$ is no greater than $n$.
\end{pr}
\begin{proof}
Let $\xi \in L_\alpha \cap L_\beta$ be a holomorphic quadraric differential. Since it is divisible by both $\alpha$ and $\beta$, it vanishes at each point $z$ at least up to order $\max\{{\ord_z}\alpha,{\ord_z}\beta\}$. The meromorphic function $\xi/(\alpha\ox\beta)$ has poles exactly at points $z$ in which one has ${\ord_z}\xi < {\ord_z}{\alpha} + {\ord_z}\beta$. This means that in order for this function to be nonconstant (i.~e. to have at least two poles or one double pole), so that $\xi$ could be not proportional to $\alpha\ox\beta$ and hence $\alpha$ and $\beta$ be linked, they must have at least two common zeroes.
\end{proof}

\begin{corol}
\label{gen-three}
Two holomorphic 1-forms on a genus two curve are always coprime. If two forms on a genus three curve are linked, this curve is hyperelliptic.
\end{corol}
\begin{proof}
Indeed, if they are linked, they must have at least two common zeroes, and if they have three, then the meromorphic function $\alpha/\beta$ would have one zero and one pole.
\end{proof}

\begin{corol}
\label{genericity}
The locus of pairs $(I,\tau) \in \Gr(2,\Omega\T_S)$, where $\tau \subset H^{1,0}(S,I)$ is linked, lies inside a subvariety of codimension two.
\end{corol}
\begin{proof}
The coincidence of at least two zeroes is an analytical codimension two condition.
\end{proof}

\begin{defn}
We shall call the locus of $k$-planes ($k=2,3$) in cohomology which can be represented by the pairs or triples of coprime holomorphic forms the {\bf coprime Kapovich--Schottky locus}.
\end{defn}

The following Proposition can be deduced from the above considerations, but is actually an elementary computation of dimensions.

\begin{pr}
For $k \geqslant 4$, the derivative of the polyperiod mapping $$\varkappa_k \colon \Gr(k,\Omega\T_S) \to \Gris(k,V)$$ is never surjective.
\end{pr}
\begin{proof}
The isotropic Gra\ss mannian $\Gris(k, 2g)$ is of codimension $1+2+\dots+(k-1) = k(k-1)/2$ inside the usual Gra\ss mannian $\Gr(k, 2g)$, which has dimension $k(2g-k)$, hence its dimension equals $2gk-(3k^2-k)/2$. The $k$-th Gra\ss mannian bundle of the Hodge bundle has dimension $3g-3 + k(g-k) = gk +3g - k^2-3$. In order for the polyperiod mapping to have surjective derivative, one must have $gk + 3g-k^2-3 \geqslant 2gk - (3k^2-k)/2$, or equivalently $g(k-3) - (k^2-k-6)/2 \leqslant 0$. For $k>3$, one can divide both sides by $k-3$ to obtain $g - (k+2)/2 \leqslant 0$. Since $g\geqslant k$, this implies $k-(k+2)/2 \leqslant 0$, or $k/2 \leqslant 1$, which is impossible whenever $k > 3$.
\end{proof}

The problem of determination of the Kapovich--Schottky locus for $k > 3$ may be interesting, but cannot be covered by the generalization of Kapovich's method.

\subsection{Sheaf-theoretic appearance of the obscurant subspaces}
\begin{defn}
Let $\tau \subset H^0(K_S)$ be an $m$-dimensional subspace spanned by holomorphic 1-forms, which we view as an injective homomorphism of sheaves $T \to \cO\ox\tau^* \approx \cO^m$. Its cokernel is called the {\bf normal sheaf} of $\tau$, and is denoted by $\nu_\tau$.
\end{defn}

The reason for the name is as follows. Suppose $A$ is an abelian surface, and $\iota \colon S \to A$ is a holomorphic mapping with at worst normal crossings from a smooth curve, the image of which is not contained in any elliptic curve. Let $\tau = \iota^*H^0(\Omega^1_A)$ be the space of restrictions of holomorphic 1-forms on $A$ to $S$. Then $\nu_\tau$ is precisely the normal bundle of $S$ inside $A$. Note that in this case the adjunction formula implies that the normal bundle is isomorphic to the canonical bundle. This can be generalized as follows:

\begin{pr}[\cite{BSY}]\label{bun-canon}
Provided the forms $\alpha_i$ spanning $\tau$ have no zero in common, the sheaf $\nu_\tau$ is a rank $m-1$ vector bundle with determinant isomorphic to the canonical bundle of $S$.
\end{pr} 
\begin{proof}
One can give a basis of $g-1$ holomorphic section of this sheaf at every point. The short exact sequence $T \to \cO^m \to \nu_\tau$ gives an isomorphism $T \ox \Lambda^{m-1}\nu_\tau \cong \Lambda^m\cO^m = \cO$, hence $\Lambda^{m-1}\nu_\tau \cong K$.
\end{proof}

The short exact sequence of sheaves $T \to \cO\ox\tau^* \to \nu_\tau$ gives rise to the long exact sequence. Since we are interested in the case $m>1$ (and hence necessarily $g>1$), $H^0(T) = 0$, and the sequence reads
$$H^0(\cO)\ox\tau^* \to H^0(\nu_\tau) \to H^1(T) \to H^1(\cO)\ox\tau^* \to H^1(\nu_\tau).$$ It gives the following characterization of the obscurant subspace:

\begin{pr}
The obscurant subspace of $\tau$ is the Serre dual of the first cohomology space $H^1(\nu_\tau)$ of the normal bundle of $\tau$.
\end{pr}
\begin{proof}
The Serre dual of $H^1(\nu_\tau)$ is the kernel of the map $H^1(\cO)^* \ox \tau \to H^1(T)^*$ dual to the substitution of tangent vectors, which by Serre duality is the map $\tau\ox H^0(K) \to H^0(K^2)$ given by the symmetric multiplication of forms.
\end{proof}

Note also that the Euler characteristic $\chi(\nu_\tau) = (3-m)(g-1)$ depends only on the dimension of $\tau$. In most cases the connecting homomorphism is zero. However, in certain interesting cases it is not.

\begin{pr}
The range of the connecting homomorphism $H^0(\nu_\tau) \to H^1(T)$ is precisely the space of first-order deformations of the curve which preserve periods of all 1-forms in $\tau$.
\end{pr}
\begin{proof}
It equals the kernel of the map $H^1(T) \to H^1(\cO)\ox\tau^*$, Serre dual to the multiplication map $\tau\ox H^0(K) \to H^0(K^2)$, hence is precisely the annihilator of quadratic differentials divisible by any $\alpha \in \tau$.
\end{proof}

\begin{expl}Let $X$ be a curve on an abelian suface $A$, and $\tau$ be the plane of restrictions to $X$ of holomorphic forms on $A$. The canonical bundle is isomorphic to normal, so any 1-form gives a variation of the curve inside a surface. When this 1-form is a restriction of a 1-form on the surface, the variation of the curve is just a parallel shift inside a torus, and hence changes not the complex structure. Otherwise it gives a nonzero class in $H^1(T)$.
\end{expl}

Now we restrict our attention to the case of two holomorphic 1-forms. 

\begin{defn}
Let $\alpha,\beta$ be two forms spanning a plane $\tau$. By its {\bf overlap} we shall mean the divisor $$Y = Y_\tau = \sum p_iz_i,~~~~p_i = \min\{\mathrm{ord}_{z_i}\alpha,\mathrm{ord}_{z_i}\beta\}.$$ It is clear that $Y_\tau$ only depends on $\tau$ and not on a choice of a basis of it.
\end{defn}

If overlap is empty, the sheaf $\nu = \nu_\tau$ is invertible and isomorphic to the canonical bundle by Proposition \ref{bun-canon}. In general, we have a mapping $\cO\ox\tau \to K$ given by $\beta \oplus -\alpha$, which annihilates the range of the map $T \arr{\alpha\oplus\beta} \cO\ox\tau^*$ and hence factorizes through its cokernel $\cO\ox\tau^* \to \nu$. If $Y_\tau \neq 0$, the natural mapping $\nu_\tau \to K$ has both kernel and cokernel. What is true is the following

\begin{pr}
Let $\tau\subset H^0(K)$ be a 2-plane, and $Y$ its overlap. Then its normal sheaf includes into the short exact sequence $\cO_Y \to \nu_\tau \to K(-Y)$, where $\cO_Y$ is a torsion sheaf.
\end{pr}
\begin{proof}
Any coherent sheaf on a curve has the torsion subsheaf with locally free quotient (i.~e. vector bundle). In this case, it is supported on the overlap.
\end{proof}

Fix an isomorphism $\nu \simeq \cO_Y \oplus K(-Y)$. Then one has an exact sequence
$$0 \to \C^2 \to H^0(\cO_Y) \oplus H^0(K(-Y)) \to H^1(T) \to H^1(\cO\ox\tau^*) \to H^1(K(-Y)) \to 0,$$

so the obscurant subspace is the dual of $H^1(K(-Y))$, which by Serre duality is $H^0(K\ox K^*(Y)) = H^0(\cO(Y))$. If $A$ and $B$ are divisors of zeroes of the forms $\alpha, \beta \in \tau$, then the latter space can be identified with $H^0(K^2(-A-B+Y))$. This is obviously the space of quadratic differentials with enough zeroes to be divisible by $\alpha$ and $\beta$, which agrees with the description of the obscurant subspace for two forms from the proof of the Proposition \ref{krichever}.

\begin{pr}\label{genericity-three}
The set of pairs $(I, \tau) \in \Gr(3,\Omega\T_S)$ s.~t. the triple $\tau \subset H^{1,0}(S,I)$ is linked, lies within a subvariety of codimension at least one.
\end{pr}
\begin{proof}
We know that linked triples are precisely the triples where the rank of $H^1(\nu_\tau)$ jumps. By semicontinuity, it is enough to find a pair $(I,\tau)$ in which this rank equals three. But this follows from a well-known refinement of the Max Noether's theorem:

\begin{fact}[Noether's theorem on quadratic differentials]Let $S$ be a non-hyperelliptic curve of degree greater than two. Then there exist three holomorphic 1-forms $\alpha,\beta,\gamma\in\Omega(S)$ s.~t. the space $H^0(K^2)$ is spanned by quadratic differentials of the form $\alpha\ox-$, $\beta\ox-$, $\gamma\ox-$.
\end{fact}
\begin{proof}
For the proof, see \cite[III.11.20, p.~149]{FK}.
\end{proof}
\end{proof}

\subsection{Counterexamples to the Hejhal--Thurston theorem for pairs and triples}
\label{htht-counter}
In this subsection, we present a number of examples of linked pairs and triples, in which the differential of the polyperiod mapping fails to be surjective, in order to give the geometric flavour of these exceptional cases.

\paragraph{Ramified covers of low-genera curves.} It follows from the Corollary \ref{one-to-one} that in order to present a pair $(S,\tau)$ of a curve and a linked triple of abelian differentials it is enough to present a curve with a triple of abelian differentials which possess nontrivial deformations, isoperiodic for each of these differentials. The simplest case is a ramified cover $\pi \colon S \to C$ of a genus three curve $C$ together with a triple $\tau = \pi^*\Omega(C) \subset \Omega(S)$. Any variation of the ramification locus gives a deformation of the curve $S$, and periods of the pullbacks of the 1-forms on the curve $C$ stay the same (since the curve $C$ itself is unchanged). Similar construction gives an example of a linked pair. Let $C'$ be a genus two curve, and $\pi' \colon S \to C'$ is a double cover ramified at $2n > 2$ points. Then the pullback of any holomorphic 1-form on $C'$ has $2n+2\times 2$ zeroes, hence the genus of $S$ equals $g=n+3$. Variations of the ramification locus produce a $2n$-dimensional family of deformations isoperiodic for all the differentials from ${\pi'}^*\Omega(C')$, which is greater than $g-2 = n+1$ predicted by the Corollary \ref{one-to-one}.

\paragraph{Hyperelliptic genus three curves.} We have noticed in the Corollary \ref{gen-three} that if a genus three curve carries two linked 1-forms, it is necessarily hyperelliptic. Now we show the converse: any hyperelliptic genus three curve has a pair of linked 1-forms, and such pairs (or rather planes spanned by such pairs) are parametrized by a rational curve.

The canonical map for a hyperelliptic curve of genus three sends it two-to-one onto the Veronese curve in $\P^2$, i.~e. a quadric. Let $a$, $b$, $c$ and $d$ be four points on this quadric. The canonical map has the characteristic property that each hyperplane (here, projective line) cuts out a canonical divisor, i.~e. corresponds to an abelian differential up to scaling. Let us denote such a differential corresponding to the line $ab$ by $\omega_{ab}$, etc. Then the differentials $\omega_{ab}$ and $\omega_{ad}$ are linked: they both divide the quadratic differential $\omega_{ab}\ox\omega_{ad}$ and the quadratic differential $\omega_{ab}\ox\omega_{cd}$ as well. The plane spanned by such a pair depends only on the point $a$ where the corresponding lines intersect, hence such planes are parametrized by the Veronese curve, which is rational.

\paragraph{Generic genus four curves.} The canonical image $C$ of a generic genus four curve is an intersection of a quadric $Q$ and a cubic in $\P^3$. The lines on the quadric are the same as the trisecants of $C \subset \P^3$. Let $a$ and $c$ be two lines on $Q$ from one family and $b$ and $d$ two lines from the other, so that $abcd$ is a spatial quadrilateral. Again, if the lines $a$ and $b$ intersect, we shall denote by $\omega_{ab}$ an abelian differential (unique up to scaling) with zero locus cut out by the plane $ab$. Then the differentials $\omega_{ab}$ and $\omega_{ad}$ are linked: they both divide the quadratic differential $\omega_{ab}\ox\omega_{ad}$ and the quadratic differential $\omega_{ab}\ox\omega_{cd}$ as well. The plane spanned by such a pair again depends on the line $a$ only, hence the planes of linked pairs for a generic genus four curves are parametrized by the variety of lines on a quadric surface, i.~e. two rational curves.

\paragraph{Sections of abelian threefolds.} Any threefold can be embedded into $\P^7$, and hence any abelian threefold $A$ possesses lots of curves cut out by sections by different $\P^5$s (precisely, parametrized by $\Gr(6,8)$, which has dimension $12$). If $S \subset A$ is such a section, then $\tau = H^{1,0}(A)|_S \subset H^{1,0}(S)$ is linked, since any other section has holomorphic forms with the same periods, and if $\tau$ were coprime, the Corollary \ref{one-to-one} would imply that all of these sections would be trivial deformations of $S$, i.~e.~parallel translations inside the torus. They only have three-dimensional space, however.

\section{Input from the ergodic theory}
\label{ratner}

\subsection{Moore's theorem}
The current picture is the following. We have (poly)period mappings $$\varkappa_k \colon \Gr(k,\Omega\T_S) \to \Gris(k,V)$$ ($k = 1, 2, 3$), and their images are the Kapovich--Schottky loci. Over some subsets of these loci, which we call coprime Kapovich--Schottky loci and which are open not just within these loci but even within the isotropic Gra\ss mannian, their fibers are of dimension $2g-3$, $g-2$ and $0$, respectively. For $k=1$, all the fibers are smooth of dimension $2g-3$; however, for $k=2$ or $3$ there are exceptional fibers of greater dimension. 

First, it is clear that the Kapovich--Schottky loci are contained in certain open subset of the isotropic Gra\ss mannian: the form $q(x) = \1\omega(x,\bar{x})$ is positive definite on any $H^{1,0}$ subspace. 

\begin{defn}
The locus of isotropic $k$-planes in a symplectic space $(V,\omega)$ on which the form $q(x) = \1\omega(x,\bar{x})$ is positive definite, is called the {\bf Hodge--Riemann Gra\ss mannian} and denoted by $\Gris^+(k,V)$.
\end{defn}

Much more can be said by the means of the ergodic theory. The space of all complex structures $\Cx(S)$ is acted by not merely the connected component of the diffeomorphism group, but the whole diffeomorphism group. Therefore the quotient $\MCG(S) = \D(S)/\D^0(S)$, called the {\it mapping class group}, acts on the Teichm\"uller space, and, via pullbacks, on the cohomology. On the space $V$ it acts through its quotient $\Sp(2g,\Z)$, preserving the intersection pairing, and hence the isotropic Gra\ss mannians $\Gris(k,V) \subset \Gr(k,V)$ and the Hodge--Riemann Gra\ss mannians $\Gris^+(k,V)$. If $f\in\D(S)$ and $\alpha$ is a holomorphic 1-form w.~r.~t. complex structure $I$, then $f^*\alpha$ is holomorphic w.~r.~t. $f^*I$, so the Hodge bundle is also preserved under the $\MCG(S)$-action, and the polyperiod mappings are $\MCG(S)$-equivariant. Therefore the Kapovich--Schottky loci and the exceptional loci within those are $\Sp(2g,\Z)$-invariant.

\begin{lm}
The Hodge--Riemann Gra\ss mannian $\Gris^+(k,2g)$ is acted upon transitively by the group $\Sp(2g,\R)$, and is isomorphic as a homogeneous space to $\Sp(2g,\R)/\left(\rU(k)\x\Sp(2g-2k,\R)\right)$.
\end{lm}
\begin{proof}
Let $\tau \in \Gris^+(k,2g)$ be some isotropic plane. It has no real vectors, since the form $q(x) = \1\omega(x,\bar{x})$ would vanish on it. Therefore the span $\langle \tau, \bar{\tau}\rangle$ has dimension $2k$, and is the complexification of a real subspace $U = U(\tau) \subset \R^{2g}$. Since the subspace $\tau$ was positive, the subspace $U(\tau)$ is symplectic. The subspace $\tau \subset U(\tau)\ox\C$ gives rise to a complex structure operator $J$ on the real space $U(\tau)$, which preserves the symplectic form (i.~e. $\omega(Jx,Jy) = \omega(x,y)$) and has the symmetric form $g(x,x) = \omega(x,Jx)$ positive definite. Therefore the Hodge--Riemann Gra\ss mannian $\Gris^+(k,2g)$ is the same thing as the Gra\ss mannian of $2k$-dimensional real symplectic subspaces in $\R^{2g}$ endowed with a suitable complex structure operator. Such pairs are acted upon by the real symplectic group $\Sp(2g,\R)$ transitively, and the stabilizer of any pair is isomorphic to the subgroup $\rU(k) \x \Sp(2g-2k,\R)$.
\end{proof}

\begin{thrm}[C.~Moore]Let $G$ be a semisimple Lie group, $\Gamma \subset G$ a lattice, and $H$ a noncompact Lie subgroup in $G$. Then $\Gamma$ acts ergodically on $G/H$. In particular, the group $\Sp(2g,\Z)$ acts ergodically on the Hodge--Riemann Gra\ss mannian $\Gris^+(k,2g) \approx \Sp(2g,\R)/\left(\rU(k)\x\Sp(2g-2k,\R)\right)$ for $k < g$.
\end{thrm}

From this, one can conclude:

\begin{pr}\label{ksch-dense}
The Kapovich--Schottky loci for $k=1,2,3$ and $g>k$ are dense in the Euclidean topology on the Hodge--Riemann Gra\ss mannian.
\end{pr}
\begin{proof}
Indeed, they contain the coprime Kapovich--Schottky loci, which are open by Proposition \ref{krichever} and Corollaries \ref{genericity} and \ref{genericity-three}, and are invariant under an ergodic action by Moore's theorem.
\end{proof}

\begin{pr}
\label{three-diff}
Let $(S, I)$ be a complex curve of genus at least four, and $\tau \subset H^{1,0}(S,I)$ be a generic three-dimensional subspace. There exist a neighborhood $U \subset \Gris(3,2g)$ containing $\tau$ s.~t. for any $\tau' \in U$ there exists a unique complex structure $I' = I(\tau')$ s.~t. one has $\tau' \subset H^{1,0}(S,I')$. In other words, deformation of a generic triple of abelian differentials determines a unique local deformation of a complex structure, so that the deformed cohomology classes would be represented by abelian differentials in the deformed complex structure.
\end{pr}
\begin{proof}
Immediate from the density of the Kapovich--Schottky locus (Proposition \ref{ksch-dense}) and local bijectivity of the triperiod mapping $\varkappa_3 \colon \Gr(3,\Omega\T_S) \to \Gris(3,V)$ in generic point (Propositions \ref{krichever} and \ref{genericity-three}).
\end{proof}

We await for someone to qualify the word `generic' by applying Ratner theory to classify the exceptional orbits in spirit of the original Kapovich's proof.

\subsection{$\SL(2,\R)$-action on the total space of the Hodge bundle}
The following consideration allows to simplify considerably Kapovich's proof of the Haupt--Kapovich theorem for the case of one class.

Let us remind that instead of the projectivization of the Hodge bundle $\P(\Omega\T_S)$ on the left and the domain $\P^+(V) \subset \P(V)$ in the projective space on the right, Kapovich considered the bundle of unit norm vectors in the Hodge bundle mapping into the unit hyperboloid in the cohomology space. As a homogeneous space, the latter is isomorphic to $\Sp(2g, \R)/\Sp(2g-2, \R)$, and fibers over $\P^+ \approx \frac{\Sp(2g,\R)}{\rU(1)\x\Sp(2g-2,\R)}$ into circles. The corresponding $\rU(1)$-action on the space of unit vectors in the Hodge bundle is multiplication of holomorphic 1-forms by the complex numbers of unit norm, and the period map is equivariant w.~r.~t. this action. Hence the image of the period map is invariant under $\rU(1)$-action on the right, i.~e. is the preimage of a subset in the domain $\P^+(V)$ (indeed, this subset is the Kapovich locus).

But the total space of the Hodge bundle possesses an action of a larger group $\SL(2,\R)$. It is well-known in the Teichm\"uller theory, see e.~g. \cite{McM1}.

\begin{defn}
Let $(X,\alpha)$ be a Riemannian surface with a holomorphic 1-form. Let us cut it along straight segments between the zeroes of $\alpha$, so that the complement $Q$ becomes connected and simply connected. Then path integration of $\alpha$ along paths within $Q$ establishes the {\it developing map} $f \colon Q \to \C$ s.~t. $\alpha|_Q = f^*(dz)$. Its image is a polygon with pairs of parallel sides corresponding to cuts between zeroes, and $(X,\alpha)$ may be reconstructed by gluing the parallel sides of the polygon $f(Q)$.

Now, let us consider $\C$ as $\R^2$ (for an algebraic geometer, this is something methodologically horrible, but it is what is to be done). The group $\SL(2,\R)$ acts on $\R^2$ and hence on its subsets. Since it preserves the parallel sides of the polygons, it lifts to the action on the pairs $(X,\alpha)$. One could check that this action is independent on choice of cuts.
\end{defn}

On the other hand, the group $\SL(2,\R) \cong \Sp(2,\R)$ acts on the unit hyperboloid $\Sp(2g,\R)/\Sp(2g-2,\R)$ with quotient $\frac{\Sp(2g, \R)}{\Sp(2, \R) \x \Sp(2g-2, \R)}$, which is the Gra\ss mannian of real symplectic 2-planes. In what follows, we shall denote the Gra\ss mannian of real symplectic $2k$-planes within $2g$-dimensional symplectic space by $\Grsymp(2k,2g)$.

\begin{pr}\label{sl2-equiv}
The period map is equivariant w.~r.~t. the $\SL(2,\R)$-action.
\end{pr}
\begin{proof}
Follows from the construction.
\end{proof}

Hence the image of the period map is $\Sp(2,\R)$-invariant, and hence is the preimage of a subset in the homogeneous space $\Grsymp(2,2g)$, the Gra\ss mannian of symplectic 2-planes. This subset is, by $\MCG(S)$-equivariance of the period map, is also invariant under $\Sp(2g, \Z)$-action. Now we can concentrate on the classification of $\Sp(2g, \Z)$-orbits in the symplectic Gra\ss mannian instead of the hyperboloid (what is what Kapovich did). This turns out to be a simpler task.

\subsection{Ratner theory}
We know that the coprime Kapovich--Schottky loci for $k=1,2,3$ are open dense, i.~e. a generic class (pair, triple of classes subject to Hodge--Riemann relations) can be realized by holomorphic 1-forms in some complex structure. Kapovich went further in order to understand the complement of the Kapovich locus for $k=1$. A tool which allows to understand the invariant closed subsets of certain ergodic actions, including ours, is the Ratner theory.

\begin{thrm}[Ratner \cite{R91}, \cite{R95}]
Let $G$ be a reductive algebraic group, $U \subset G$ a connected subgroup generated by unipotent elements (i.~e.~ones with unipotent adjoint action on the Lie algebra of $G$), and $X = G/U$. Let $\Gamma \subset G$ be a lattice. For any $x = gU \in X$ the connected component of the closure of the orbit $\Gamma x \subset X$ is equal to the orbit $H^g x \subset X$ for some Lie subgroup $H^g \subset G$ s.~t. the intersection $H^g \cap \Gamma \subset H^g$ is a lattice in it.
\end{thrm}

\begin{pr}\label{discrete-or-dense}
An $\Sp(2g,\Z)$-orbit in $\Grsymp(2,2g)$ is either discrete or dense.
\end{pr}
\begin{proof}
The lattice $\Sp(2g, \Z) \subset \Sp(2g, \R)$ is a lattice in a reductive group, and the group $U = \Sp(2, \R) \x \Sp(2g-2, \R)$ is generated by unipotent elements. Hence the Ratner theorem applies, and the connected component of the closure of any $\Sp(2g, \Z)$-orbit is an orbit of an intermediate subgroup $H \colon \Sp(2, \R) \x \Sp(2g-2, \R) \subseteq H \subseteq \Sp(2g, \R)$. It follows from Dynkin's classification of maximal subgroups \cite{D} that a subgroup $\Sp(2, \R) \x \Sp(2g-2, \R) \subset \Sp(2g, \R)$ is maximal, hence one has either $H = \Sp(2g, \R)$ (and the closure of the orbit is the whole $\Grsymp(2,2g)$, i.~e. the orbit is dense), or $H = \Sp(2, \R) \x \Sp(2g-2, \R)$ (and the connected component of the closure of the orbit is a single point, i.~e. the orbit is discrete).
\end{proof}

Let us also recall the following theorem by Margulis.

\begin{thrm}[Margulis \cite{M}]Suppose that $G$, $G'$ are simple real algebraic Lie groups such that their complexifications do not have isomorphic Lie algebras. Then any lattice $\Gamma \subset G \x G'$ is reducible, i.e. $\Gamma \cap G \subset G$, $\Gamma \cap G' \subset G'$ are lattices, and $\Gamma$ splits into a product of those.
\end{thrm}

\begin{pr}\label{margulis}
Let $g > 2$. Suppose an $\Sp(2g, \Z)$-orbit within a symplectic 2-planes Gra\ss mannian $\Grsymp(2,2g)$ is discrete. Then the symplectic subspaces preserved by the group $\Sp(2, \R) \x \Sp(2g-2, \R)$ are spanned by the integral vectors.
\end{pr}
\begin{proof}
By Ratner's theorem, in this case the intersection of $\Sp(2g, \Z)$ with $\Sp(2, \R) \x \Sp(2g-2, \R)$ is a lattice. By Margulis' theorem, its intersection with both factors $\Sp(2, \R)$ and $\Sp(2g-2, \R)$ are lattices, i.~e. isomorphic to $\Sp(2, \Z)$ and $\Sp(2g-2, \Z)$, respectively.
\end{proof}

In cases $k=2$ or $3$, Ratner theory might be also applied to determine the orbits with fibers of dimension larger than predicted over them (that is, the non-coprime part of Kapovich--Schottky locus). We hope to achieve it in a subsequent paper. Let us formulate a proposition which would serve as a backlog for it.

\begin{pr}
Let $g=3$ or $g>4$, and $\tau \in \Gris^+(2,2g)$ be a Hodge--Riemann plane. Then either the real four-dimensional symplectic subspace $$U(\tau) = \langle\tau\cup\bar{\tau}\rangle^{\mathrm{Gal}(\C:\R)}$$ is spanned by the integral vectors, or the image of the orbit $\Sp(2g,\Z)U(\tau) \subset \Grsymp(4,2g)$ is dense.
\end{pr}
\begin{proof}
Apply the combination of Ratner's and Margulis' theorems from the proof of the Proposition \ref{margulis} to the case of a maximal subgroup $\Sp(4,\R)\x\Sp(2g-4,\R) \subset \Sp(2g,\R)$.
\end{proof}

Note that the density of the orbit $\Sp(2g,\Z)U(\tau) \subset \Grsymp(4,2g)$ within the symplectic Gra\ss mannian does not guarantee that the orbit $\Sp(2g,\Z)\tau \subset \Gris^+(2,2g)$ within the Hodge--Riemann Gra\ss mannian is dense: it would imply that it intersects, and hence falls into, the coprime Kapovich--Schottky locus. Quite on the contrary, non-dense orbits above with dense projections, corresponding to the non-representable pairs or non-coprime part of the Kapovich--Schottky locus, are of interest to us.

\section{Elliptic classes}\label{el}

The current picture for one class is the following. We have dense orbits, and we have discrete orbits. Classes from the dense orbits (e.~g. ones with the dense subgroup of periods) belong to the Kapovich locus, since it is open. Therefore now we need to understand which classes with discrete orbits belong to it. This knowledge would be useful in the case of several classes, too.

\subsection{Motivation: curves on tori}
\begin{defn}
Let $V$ be a real vector space, and $\tau \subset V_\C$ a subspace which contains no nonzero real vectors. Then its {\bf real part} is the subspace $U(\tau) =  \mathrm{span}\langle\tau\cup\bar{\tau}\rangle^{\mathrm{Gal}(\C:\R)} \subset V$.
\end{defn}

\begin{defn}
Let $V_\Z$ be a rank $2g$ lattice equipped with the standard symplectic form $\omega = \sum_{i=0}^{g-1}dx_{2i} \wedge dx_{2i+1}$. A Hodge--Riemann subspace $\tau \subset V_\C$ is called {\bf elliptic} if its real part $U(\tau)$ is a rational subspace, i.~e. its intersection $U_\Z(\tau) = U(\tau) \cap V_\Z$ with $V_\Z$ has rank $2\dim\tau$.
\end{defn}

\begin{pr}[partial Abel--Jacobi construction]\label{AJac}Let $(S,I)$ be a genus $g>1$ curve, and $\tau \subset H^{1,0}(S,I)$ an elliptic subspace of dimension $\dim\tau = k$. Then there exists a $k$-dimensional complex torus $A = A(\tau)$ with $H^1(A,\Z)$ canonically identified with $U_\Z(\tau)$ and a holomophic map $f\colon(S,I) \to A$ s.~t. one has $\tau=f^*H^{1,0}(A)$, and the induced map $f^* \colon H^1(A,\Z) \to H^1(S,\Z)$ is the tautological inclusion $U_\Z(\tau) \to H^1(S,\Z)$.
\end{pr}
\begin{proof}
The usual Albanese torus of a curve $(S,I)$ arises as the quotient $H_1(S,\R)$ by $H_1(S,\Z)$ endowed with complex structure given by the distribution of $(1,0)$-vectors in the cotangent bundle, which coincide with the subspace $H^{1,0}(S,I)$. If $U(\tau)$ is a real subspace in cohomology $H^1(S,\R)$ spanned by the lattice $U_\Z(\tau)$, then the dual space $U(\tau)^*$ is the quotient of $H_1(S,\R)$, and the dual lattice $U_\Z(\tau)^*$ (i.~e. the lattice of covectors which pair integrally with the vectors in $U_\Z(\tau)$) is the projection of the integral vectors from $H_1(S,\Z) \subset H_1(S,\R)$. Therefore there exists a surjective map of tori $\mathrm{Alb}(S,I) \to U(\tau)^* \mod U_\Z(\tau)^* = A(\tau)$, where the latter is endowed with the complex structure in which the $(1,0)$-vectors are precisely the vectors from $\tau$ (the projection is holomorphic since $\tau$ is cut out by $U(\tau)\ox\C$ from $H^{1,0}(S,I)$). The universal property now follows from the universal property of the Albanese map.
\end{proof}

Notice that instead of the lattice $U_\Z(\tau)$ one could have chosen a smaller intermediate lattice $L \supset U_\Z(\tau)$ which yet spans $U(\tau)$. However, its dual $L^* \subset U(\tau)^*$ would in this case be a larger lattice, and the quotient $U(\tau)^* \mod L^*$ would be a torus obtained from $A(\tau)$ by quotienting out a finite subgroup. Hence the torus $A(\tau)$ is the largest one from which one can pull back the holomorphic forms from $\tau$.

\begin{defn}
We say that a sublattice $L \subset \Z^{2g}$ is {\bf complete} if it is cut out by some vector subspace in $\mathbb{Q}^{2g} \cong \Z^{2g}\ox\mathbb{Q}$. Equivalently, $L$ is complete iff for any $v \not\in L$, $n>0$ one has $nv \not\in L$.
\end{defn}

\begin{expl}
Let $S$ be a genus two curve with an involution $\iota$, which has two fixed points, and let $\{e_0, f_0, e_1, f_1\}$ be the standard basis in the integral cohomology. The involution acts as $\iota^*e_0 = -e_1$, $\iota^*f_0 = -f_1$, so the fixed vectors are $e_0 - e_1$ and $f_0 - f_1$. This is precisely the inverse image of integral homology classes on the elliptic curve $S/\iota$; one has $\omega(e_0-e_1,f_0-f_1) = \omega(e_0,f_0) + \omega(e_1,f_1) = 1+1 = 2$. Yet, there are no vectors $v \in H^1(S, \Z)$ s.~t. $v \not\in \mathrm{span}\{e_0-e_1,f_0-f_1\}$ while $2v\in\mathrm{span}\{e_0-e_1,f_0-f_1\}$, since this sublattice is complete.
\end{expl}

\begin{expl}
Let $S$ be a genus three curve, which is a double unramified cover of a genus two curve $S'$, which is in its turn a double cover of an elliptic curve $E'$ ramified in two points. The the pullback map $p \colon H^1(E',\Z) \to H^1(S,\Z)$ factorizes through the map $\colon H^1(S', \Z) \to H^1(S, \Z)$, and since it is the pullback along an unramified cover, vectors in its image are divisible by the degree of the cover. Hence the image of the map $H^1(E',\Z) \to H^1(S,\Z)$ is not a complete sublattice. However, applying the partial Abel--Jacobi construction to the line $p^*H^{1,0}(E') \subset H^1(S,\C)$ yields a map into an elliptic curve $E$, which doubly covers the elliptic curve $E'$, and the inverse image of $H^1(E,\Z)$ within $H^1(S,\Z)$ would be precisely the lattice of integral vectors $\frac12p^*H^1(E',\Z) \subset H^1(S,\Z)$. So we see that the partial Abel--Jacobi map related to a complete sublattice is indeed the mapping to the largest possible torus, from which we can obtain a desired subspace in cohomology via pullback.
\end{expl}

\begin{expl}
Let $S \arr{\pi} S'$ be a genus three curve, which is a double unramified cover of a genus two curve $S'$. Namely, let $e_0', f_0', e_1', f_1'$ be the simple loops representing a standard basis in homology $H_1(S',\Z)$, and $S$ is obtained from two copies of $S'$ cut along the loop $e_1'$ and glued criss-cross. Then homology $H_1(S,\Z)$ has a standard basis $\{e_0,f_0,e_1,f_1,e_2,f_2\}$ s.~t. one has $$\pi_*e_0 = \pi_*e_2 = e_0',$$ $$\pi_*f_0 = \pi_*f_2 = f_0',$$ $$\pi_*e_1 = e_1',~\pi_*f_1 = 2f_1'.$$ If $\eps_i, \phi_i$ (resp. $\eps_i', \phi_i'$) represent the dual classes in $H^1(S,\Z)$ (resp. $H^1(S',\Z)$), one has $$\pi^*\eps_0' = \eps_0 + \eps_2,$$ $$\pi^*\phi_0' = \phi_0 + \phi_2,$$ $$\pi^*\eps_1' = \eps_1,~\pi^*\phi_1' = 2\phi_1.$$ The partial Abel--Jacobi map for the incomplete sublattice spanned by these vectors is, by construction, a map to the Albanese surface of $S'$, which is a two-to-one map onto the Abel--Jacobi image $S' \to \mathrm{Alb}(S')$. However, if one takes a complete sublattice which contains this sublattice, i.~e. the sublattice spanned by $\eps_0 + \eps_2, \phi_0+\phi_2, \eps_1, \phi_1$, one gets an injective map into a double cover $\widetilde{\mathrm{Alb}(S')}$ of the Albanese surface. The reason is the Poincar\'e complete reducibility theorem: the mapping $\pi_* \colon \mathrm{Alb}(S) \to \mathrm{Alb}(S')$ admits a multivalued section (here, of degree two).
\end{expl}

\subsection{Single elliptic class: Haupt--Kapovich condition}
\begin{defn}
Let $U \subset V_\Z$ be a rank $2k$ symplectic sublattice (i.~e.~the~restriction of the form $\omega$ to $U$ is nondegenerate). Consider the cyclic submodule $\Lambda^{2k}_\Z{U} \subset \Lambda^{2k}_\Z{V_\Z}$, and let $\xi$ be its positive generator. Then the {\bf determinant} $\det U$ of $U$ equals to $\omega^{\wedge k}(\xi)$. If $\tau \subset V_\C$ is an elliptic subspace, its {\bf determinant} is the determinant of the lattice of integral vectors within its real part $U_\tau$.
\end{defn}

\begin{pr}\label{trans-indiv}
The group $\Sp(2g, \Z)$ acts transitively on the indivisible vectors in $\Z^{2g}$.
\end{pr}
\begin{proof}
Let us realize the symplectic lattice $\Z^{2g}$ with its standard symplectic form $\sum_{i=0}^{g-1}dx_{2i}\wedge dx_{2i+1}$ as the homology lattice of a genus $g$ surface with its intersection pairing. Any indivisible integral homology class on a surface can be represented by a nonseparating cycle (maybe this requires uniformization and Eels--Sampson theorem, but let us assume this is obvious). Then cutting two copies of the surface along two cycles corresponding to two indivisible vectors in homology yields two surfaces with boundary, which are homeomorphic by the classification theorem for surfaces. After gluing the boundaries back, this homeomorphism establishes an auto-homeomorphism of the surface sending one cycle to another. It acts on the homology preserving the symplectic pairing, hence gives an element of $\Sp(V_\Z,\omega)$ sending one indivisible vector to another.
\end{proof}

\begin{pr}\label{det-trans}
Let $U, U' \subset V_\Z$ be two rank two complete symplectic sublattices with determinant $d$. Then there exists an element $\delta \in \Sp(V_\Z,\omega)$ s.~t. one has $\delta U = U'$.
\end{pr}
\begin{proof}Since $\Lambda^2_\Z(U)$ is a cyclic module, there exists a basis $x,y \in U$ s.~t. $\det U = \omega(x,y)$, and similarly for $U'$. In this case, both $x$ and $y$ are indivisible.

Let $\{e_i, f_i\}_{i=0}^{g-1}$ be the standard symplectic basis for $\Z^{2g}$. Since the group $\Sp(V_\Z, \omega)$ acts transitively on indivisible vectors, one can assume that $U$ and $U'$ have a common indivisible basic vector, say $e_0$. Since $\det U =d$, any vector spanning $U$ with $e_0$ is of form $df_0 + u$, $u \in e_0^\perp$, and similarly $U'$ is spanned by $e_0$ and $df_0 + u'$. Therefore what we are looking for is a symplectic matrix $A$ s.~t. $Ae_0 = e_0$ and $A(df_0 + u) = df_0 + u' + ke_0$ (one cannot hope that we can send one basis to another, since, for example, no matrix from $\mathrm{SL}(2,\Z)$ can send the basis $\begin{pmatrix} 1 & 1 \\ 0 & 3\end{pmatrix}$ to $\begin{pmatrix} 1 & 2 \\ 0 & 3\end{pmatrix}$; yet they generate the same sublattice, which is in this case not complete though). One has $0 = \omega(f_0,e_0) - \omega(Af_0,Ae_0) = \omega(f_0-Af_0,e_0)$, hence $Af_0 = f_0 + x$, where $x \in e_0^\perp$, and the equation amounts to $df_0 + dx + A(u) = df_0 + u' + ke_0$, or simply $dx + Au = u' + ke_0$. This is an equation on vectors within the sublattice $e_0^\perp$, which boils down to $d[x] + [A][u] = [u']$, where $[y] = y \mod e_0 \in e_0^\perp/e_0$. Notice that the latter is isomorphic as a symplectic lattice to the standard $\Z^{2g-2}$, and since the sublattices $U, U'$ were complete, the vectors $u, u'$ are indivisible modulo $e_0$, that is, $[u], [u'] \in \Z^{2g-2}$ are indivisible. Hence by Proposition \ref{trans-indiv} they are related by a matrix $[A]\in \Sp(e_0^\perp/e_0,\omega)$. The mapping from the stabilizer of $e_0$ within $\Sp(V_\Z,\omega)$ to $\Sp(e_0^\perp/e_0,\omega)$ is surjective, and hence this gives a desired matrix $\delta \in \Sp(V_\Z,\omega)$ for $x=0$.
\end{proof}

\begin{pr}[topologization lemma]\label{topolo}Suppose an elliptic line $\tau \subset H^1(S,\C)$ is realizable by a holomorphic 1-form. Then any Hodge--Riemann line $\tau'$ with $U(\tau') = U(\tau)$ is realizable by a holomorphic 1-form. To put it another way, a property of an elliptic line $\tau$ of being realizable by an abelian differential is a property of its real part $U(\tau)$.
\end{pr}
\begin{proof}
If $I$ is a complex structure on $S$ s.~t. $\alpha \in H^{1,0}(S,I)$ spans the line $\tau$, the path integration yields a holomorphic map $S \to E_\tau$, where $E_\tau$ is the torus $U(\tau) \mod U_\Z(\tau)$ with the structure of an elliptic curve in which $H^{1,0}(E_\tau) = \tau$. We can consider it just as a topological map $S \to E = E_{\tau}$, choose another complex structure on $E$, given by $H^{1,0}(E_{\tau'}) = \tau'$, and pull it back to $S$ so that this topological map would be again holomorphic in the new complex structures, and the line $\tau' \subset H^1(S,\C)$ would be spanned by the pullback of the abelian differential on the elliptic curve $E_{\tau'}$. Therefore, whether the line $\tau$ is realizable or not, depends only on the real part of $\tau$.
\end{proof}

This topologization lemma is actually just a way to state in the elliptic case that the period map is $\SL(2,\R)$-equivariant (Proposition \ref{sl2-equiv}). We would like to give it a proper name since we hope it might be important in the case of a pair of classes (see Proposition \ref{topolo-pair} below).

\begin{pr}\label{hk-where-from}
An elliptic line $\tau \in \P(H^1(S,\C))$ can be realized by a holomorpic 1-form iff its determinant is greater than $1$.
\end{pr}
\begin{proof}
By the topologization lemma (Proposition \ref{topolo}), whether an elliptic line $\tau$ can be realized by an abelian differential or not, depends on its real part $U(\tau)$, which is a rational 2-dimensional subspace, only. However, by Proposition \ref{det-trans}, any two symplectic sublattices of rank two in the cohomology with equal determinants are related by the action of the mapping class group. Hence if the determinant of $\tau$ is $d$, then $\tau$ is realizable iff there exists one realizable elliptic line with determinant $d$. For genus $g$, it amounts to the topological question of finding a ramified cover of degree $d$ of a genus $g$ surface over a torus, which is always possible for $d>1$ (pick up $2g-2$ points in a torus, and construct a connected $d$-sheeted ramified cover with $d-1$ preimages over the chosen $2g-2$ points), and impossible for $d=1$ provided $g>1$.
\end{proof}

\begin{corol}One can realize any realizable elliptic line by a holomorphic 1-form on a hyperelliptic curve. 
\end{corol}
\begin{proof}
Choose the ramification divisor to be symmetric with respect to the standard involution on the elliptic curve.
\end{proof}

Let us sum up our proof of the Haupt--Kapovich theorem for one class.

\begin{pr}[Haupt, Kapovich]\label{hk-new}Let $S$ be a genus $g>2$ topological surface. A line $\langle\alpha\rangle \in \P^+(V)$ spanned by a class $\alpha \in V = H^1(S,\C)$ belongs to the Kapovich locus iff the set of periods $$\{\alpha(\gamma)\}_{\gamma \in H_1(S,\Z)} \subset \C$$ is either dense or a lattice in $\C$ with $\1\omega(\alpha,\bar{\alpha})$ greater than the covolume of this lattice (the latter conditon is called the {\normalfont Haupt--Kapovich condition}).
\end{pr}
\begin{proof}
By Proposition \ref{sl2-equiv}, the Kapovich locus in $\P^+(V)$ is the preimage of a union of $\Sp(V_\Z, \omega)$-orbits in $\Grsymp(2,V_\R)$, the symplectic 2-planes Gra\ss mannian, over which $\P^+(V)$ fibers into disks. By Proposition \ref{discrete-or-dense}, these orbits are either dense or discrete. The dense orbits correspond to the classes with dense set of periods, and, by Proposition \ref{margulis}, the discrete orbits correspond to the elliptic classes. Since by Proposition \ref{htht-strong} the Kapovich locus is open, any dense orbit intersects it and hence belongs to it; an elliptic orbit belongs to the Kapovich locus iff the Haupt--Kapovich condition is satisfied by Proposition \ref{hk-where-from}.
\end{proof}

Notice that in the case $g=2$ the group $\Sp(2,\R) \x \Sp(2,\R)$ contains non-splitting lattices, namely, lattices arising from the real quadratic orders: if $K = \mathbb{Q}\left(\sqrt{d}\right)$, one can consider $\SL(2,\cO_K)$ as a lattice in $\SL(2,\R)\x\SL(2,\R)$ via the two embeddings $K \to \R$. They correspond to the genus two curves with an abelian differential, whose trajectory under $\SL(2,\R)$-action, projected to the moduli space of curves, has two-dimensional closure, called the Hilbert modular surface. The Jacobian surfaces of such curves admit real multiplication by $K$. The other two cases for $g=2$ are the elliptic case, when the closure of the projection is a curve, and otherwise the whole moduli space. \cite{McM2}\cite[Theorem 1.5]{CDF} It is interesting whether similar lattices can occur in the group $\Sp(4,\R)\x \Sp(4,\R)$, which would correspond to certain pairs of abelian differentials on genus four curves.

\subsection{Elliptic pairs}
We now try to draw a similar sketch for two-dimensional elliptic subspaces.

\begin{pr}
Let $(L,\omega)$, $(L',\omega')$ be two rank four lattices equipped with integral symplectic forms, not necessarily standard, yet indivisible (i.~e. vectors $\omega \in \Lambda^2_{\Z}L^*$, $\omega' \in \Lambda^2_{\Z}L'^*$ are indivisible) and having equal determinants. Then there exists a linear isomorphism $L \arr{\sim} L'$ sending $\omega$ to $\omega'$.
\end{pr}
\begin{proof}
This proof has been kindly told me by Misha Verbitsky, of IMPA, Rio de Janeiro.

Choose an arbitrary linear isomorphism $L \to L'$, so one can consider $\omega$ and $\omega'$ as vectors in one and the same lattice $\Lambda^2_{\Z}(L^*)$. Wedge multiplication $\wedge\colon\Lambda^2_{\Z}(L^*) \x \Lambda^2_{\Z}(L^*) \to \Lambda^4_{\Z}(L^*)$ makes it into a quadratic lattice, after one picks up the generator of the cyclic module $\Lambda^4_{\Z}(L^*)$ (e.~g. in such a way that $\omega\wedge\omega$ becomes a positive vector; one can assume that $\omega'\wedge\omega' = \omega\wedge\omega$, since they have equal determinants). The linear group $\SL(L)$ acts as the special orthogonal group of this quadratic lattice. Eichler's theorem applies: it asserts that the indefinite special orthogonal groups acts transitively on indivisible vectors of equal positive norms \cite{E}. Hence there exists an element in $\SL(L)$ relating $\omega'$ to $\omega$; twisting the initial isomorphism $L \to L'$ by it gives an isomorphism sending one symplectic form to the other.
\end{proof}

\begin{pr}
The automorphism group $\Sp(2g,\Z)$ of the standard integral symplectic lattice $(\Z^{2g},\omega)$ acts transitively on rank four complete symplectic sublattices with determinant $d$.
\end{pr}
\begin{proof}
Let $U \subset \Z^{2g}$ be a complete rank four symplectic sublattice with determinant $d$. Since $\omega$ is indivisible, its restriction to a complete symplectic sublattice is, and hence $U$ is abstractly isomorphic to the lattice $dQ\oplus Q$, where $Q \simeq (\Z^2,\omega)$ is the standard integral symplectic lattice of rank two. Consider the unimodular sublattice $Q \subset U$. By Proposition \ref{det-trans}, there exists an element $\delta \in \Sp(2g,\Z)$ s.~t. $\delta Q = \mathrm{span}\{e_0,f_0\}$. It necessarily sends $dQ \subset U$ to a rank two complete sublattice in $\Z^{2g-2} = \mathrm{span}\{e_1,f_1,e_2,\dots,f_g\} \subset \Z^{2g}$. The stabilizer of $\{e_0,f_0\} \subset \Z^{2g}$ acts on its orthogonal complement as $\Sp(2g-2,\Z)$, hence, by Proposition \ref{det-trans}, transitively on complete rank two sublattices of determinant $d$.
\end{proof}

\begin{defn}
An elliptic pair $\tau$ is called {\bf simple} if the abelian surface $A = U(\tau)/U(\tau)_\Z$, $H^{1,0}(A) = \tau^*$ is not isogeneous to a product of elliptic curves.
\end{defn}

\begin{pr}[topologization lemma for a pair]\label{topolo-pair}Suppose a simple elliptic pair $\tau \in \Gris(2,H^1(S,\C))$ is realizable by holomorphic 1-forms. Then any simple Hodge--Riemann plane $\tau'$ with $U(\tau') = U(\tau)$ is realizable by holomorphic 1-forms.
\end{pr}
\begin{proof}
By Proposition \ref{AJac}, a realizable elliptic plane $\tau \subset H^1(S,\C)$ is the same as a map $S \to A^2$ with image an analytic curve. A deformation of the plane $\tau$ within $U(\tau)\ox\C$ is the same as a deformation of a left-invariant complex structure on an abelian surface, hence the lemma is equivalent to the following proposition: given a genus $g>1$ curve on an abelian surface, for any deformation of the complex structure on the surface, the curve may be deformed in order to stay analytic.

Let us notice that on an abelian surface a complex structure $J$ is encoded by a holomorphic $(2,0)$-form $\sigma = \sigma_J$, which is unique up to scaling. This allows us to employ the deformation construction from \cite{BDV}, yet in the other direction. Let us fix a left-invariant analytic foliation $\ell \subset TA$, and let us call a deformation of the complex structure on $A$ {\bf isofoliant} for $\ell$, if the foliation $\ell$ stays analytic, or, in terms of the form $\sigma_J$, Lagrangian. The curve $S$ may be viewed as a multi-valued section of the fibration by the leaves of $\ell$ (which makes no sense globally though). Let $J_t$ be an isofoliant deformation of $J$. The Hodge--Riemann relations (namely, the isotropicity condition) asserts that one has $\int_S\sigma_{J_t} = 0$ for any $t$, hence one has $\sigma_{J_t}|_S = d\eta_t$. The holomorphic symplectic form may be viewed as a pairing $T^*S \to \ell|_S$ at least away from the points where $S$ is tangent to $\ell$, and it sends the form $\eta_t$ to a section of $\ell|_S$. This section is the direction of the deformation $S_t$ which would make the identity $\sigma_{J_t}|_{S_t} = 0$ true not merely on homological level, but rather pointwise. On a holomorphic symplectic surface a 2-dimensional submanifold is an analytic curve iff it is Lagrangian, so this deformation would make the curve $S_t$ analytic for $J_t$ away from its points of tangency with $\ell$. At these points the curve $S_t$ is automatically analytic since the deformation $J_t$ is isofoliant for $\ell$. This proves the Proposition for isofoliant deformations; now notice that the moduli space of abelian surfaces may be navigated by broken lines of isofoliant deformations for different choices of foliations.

Nominally we proved the Proposition for any abelian surface, or any elliptic pair. However, during such a deformation curve may develop singularities, and we cannot assure ourselves that this does not occur in the case when the limiting abelian sufrace is isogeneous to a product. For example, when $S$ is a genus two curve on its Jacobian surface $A$, and $A$ approaches the product of two elliptic curves, $S$ breaks into these two elliptic curves, intersecting in one point.
\end{proof}
\begin{proof}
Let us give a more formal proof of the infinitesimal statement. A holomorphic symplectic form establishes an isomorphism of holomorphic bundles $T \to \Omega$, $v \mapsto \iota_v\sigma$. Hence one has $H^1(T_A) \cong H^1(\Omega_A) = H^{1,1}(A)$. For $A$ an abelian surface, the condition on the first order deformation $\alpha \in H^{1,1}(A)$ to stay projective is $\alpha\wedge\omega = 0$, where $\omega \in H^{1,1}(A)$ is the K\"ahler class. If $S \subset A$ is a curve, then $\alpha|_S \in H^{1,1}(S)$ is the obstruction to shifting $S$ into the corresponding deformation of $A$. However, whenever $S$ is not an elliptic curve, it represents the multiple of the hyperplane section, and $\int_S\alpha = \int_A\alpha\wedge\omega = 0$. 
\end{proof}

\begin{pr}[Haupt--Kapovich condition for an elliptic pair]\label{hk-el-pair}Let $S$ be a genus $g>2$ surface, and $\tau \subset H^1(S,\C)$ a simple elliptic pair. Then there exists a complex structure $I$ on $S$, an abelian surface $A$, and a generically injective holomorphic map $f \colon (S,I)\to A$ s.~t.~one has $f^*H^{1,0}(A) = \tau$, and any other map to an abelian surface $A'$ with this property factorizes as $S \arr{f} A\to A'$, iff the determinant $\det\tau$ is an even integer no less than $2g-2$.
\end{pr}
\begin{proof}
By Proposition \ref{topolo-pair}, one suffices to construct a single example for given numerical values of $g$ and $\deg\tau$ in order to assure existence in any such case.

Let us first notice that if $S \subset A$ is a smooth genus $g$ curve and $A \to A'$ is a degree $n$ covering, than the image of the map $S \to A \to A'$ has $(n-1)(2g-2)$ nodes. Therefore, if $f \colon S \to A$ is an injective map, it cannot factorize through any nontrivial covering $\widetilde{A}\to A$. Indeed, in such case the map $S \to \widetilde{A}$ is also injective, and $f(S) \subset A$ has nodes whenever the degree of the covering is greater than one. By adjunction, the self-intersection of a smooth genus $g$ curve on an abelian surface (i.~e. its maximal possible degree of polarization) is $2g-2$.

Let us prove a lemma.

\begin{lm}
Let $C$ be a curve, and $p,q$ two points on it. Then there exists a meromorphic 1-form with first order poles at $p$ and $q$, and regular elsewhere. Its residues at $p$ and $q$ are opposite, and this form is unique up to scaling and adding a regular 1-form.
\end{lm}
\begin{proof}
Forms with possible first order poles at $p$ and $q$ are the sections from $H^0(K_C \o \cO(p+q))$. By Riemann--Roch formula, one has $h^0(K(p+q)) - h^1(K(p+q)) = \deg K + 2 - g + 1 = g+1$. By Serre duality, one has $h^1(K(p+q)) = h^0(K \o T(-p-q)) = h^0(\cO(-p-q)) = 0$, hence $h^0(K(p+q)) = h^0(K)+1$, and hence a non-regular 1-form with poles at $p$ and $q$ exists and is unique up to scaling and adding a regular form. Its residues sum up to zero, and since there are only two of them, they are opposite.
\end{proof}

Now an idea of construction might be as follows. Consider a genus two curve $C$ in its Jacobian surface $\mathrm{Jac}(C)$, and let $\varphi \colon \mathrm{Jac}(C) \to A$ be a degree $n$ covering. Then $\varphi(C)$ is a geometric genus two curve with $n-1$ nodes. Consider the (local) moduli space of arithmetic genus $n+1$ non-cuspidal curves on $A$. The tangent space to the point corresponding to $\varphi(C)$ may be identified with the space of such meromorphic 1-forms on the normalization $\widetilde{\varphi(C)} = C$ that have first order poles at preimages of nodes with opposite residues at preimages of the same node, modulo restrictions of holomorphic 1-forms on $A$. A deformation smooths out a node iff the corresponding 1-form has pole in it; therefore by the above Lemma the first-order deformations preserving all nodes except one form a system of coordinate hyperplanes. Hence we are able to smooth out one node after another, and thus to obtain a geometric genus $g$ curve with $n-g+1$ nodes for any $g$ between $2$ and $n+1$, which lies on an abelian surface with polarization $2n$.

In principle, these nodal curves may be lifted to a nontrivial covering of $A$, and this is obviously the case for $n = 0$, since our original genus two curve came from its Jacobian surface. However, for $n>0$ there exist nontrivial deformations of these curves parametrized by holomorphic 1-forms not coming from $A$, and a generic deformation would not lift to any cover.
\end{proof}

The above theorem may be regarded as an analogue of the classical theorem of Severi \cite[Anhang F]{S}, which states that a degree $d$ plane curve may have precisely $n$ nodes and no other singularities for any $0 \leqslant n \leqslant (d-1)(d-2)/2$. It does not seem to be present in the available literature, though.

We expect that, much like in the case of a single class, these elliptic pairs are the most extreme case of pairs realizable by abelian differentials: probably, any other realizable pair can be herded by the $\Sp(2g,\Z)$-action into a tiny neighborhood of an elliptic pair.

\subsection{$\mathfrak{sp}(4,\R)$-action?}
The topologization lemma (Proposition \ref{topolo}) is a simple instant of a deep phenomenon, which is foundational for Teichm\"uller theory, namely, the $\SL(2,\R)$-action on the total space of the Hodge bundle (which is known as the {\it moduli space of abelian differentials}). A similar, yet weaker, topologization lemma for pairs (Proposition \ref{topolo-pair}) raises the following question: whether the fiberwise Gra\ss mannian $\Gr(2,\Omega\T_S)$ (which in this context might be called the {\bf moduli space of abelian bidifferentials}) possesses a natural action of the group $\Sp(4,\R)$, which would be compatible with the $\Sp(4,\R)$-action on the Teichm\"uller space for abelian surfaces (i.~e. the Siegel upper half-space $\Sp(4,\R)/\rU(2)$) for elliptic pairs?

In a way, the answer is obviously negative. For $g = 2$ the fiberwise Gra\ss mannian $\Gr(2,\Omega\T_g)$ coincides with $\T_g$ itself, which maps into the Siegel upper half-space of principally polarized abelian surfaces as a quotient $\T_2/\MCG^0(2)$ (here $\MCG^0(g)$ is the kernel of the natural map $\MCG(g) \to \Sp(2g,\Z)$, the group of classes of mappings which act trivially on the homology, also known as the {\it Torelli group}). The Birman--Powell theorem \cite{HM*} asserts, in particular, that for $g=2$ this group is generated by Dehn twists along homologically primitive cycles, hence the vanishing cycle for a family of genus two curves over a punctured disk with central fiber omitted by the Torelli map is homologically primitive, i.~e. this family compactifies by a reducible curve consisting of two elliptic curves meeting in one point. Therefore the Torelli map omits the locus of surfaces which fall apart into a product of elliptic curves. However, it is possible to pull back the local action of the Lie algebra $\mathfrak{sp}(4,\R)$. One might wonder if such an action exists on the moduli space of abelian bidifferentials on any genus curve. 

It might be constructed ad hoc for smooth genus three curves lying on an abelian surface. Indeed, by a theorem of Barth \cite[Proposition (1.8)]{B} a smooth genus three curve can be embedded into an abelian surface iff it is bielliptic, i.~e. a two-sheeted cover of an elliptic curve. For a curve $C$ on an abelian surface $A^2$ this elliptic curve $E = E(C \subset A^2)$ arises as the kernel of the map $\mathrm{Jac}(C) \to A^2$. If one fixes the surface $A^2$, a deformation of $C$ within $A^2$ (which has $g-2 = 1$ degree of freedom) is governed by the deformation of the kernel $E$. The other way around, if one fixes the kernel $E$, a deformation of $C$ (given by a variation of the quadruple of ramification points) gives rise to a deformation of the quotient $A^2$, to where $C$ belongs (at least whenever $A^2$ is simple). Therefore, for fixed $E = E(C \subset A^2)$, the local action of the Lie algebra $\mathfrak{sp}(4,\R)$ on the Siegel upper half-plane lifts to its local action on possible ramification divisors on $E$ (defined up to a shift, of course), and hence on the local moduli space of pairs $\{C \subset A^2\}$ (which is an orbit within the moduli space of abelian bidifferentials for genus three).

\section{Intrinsic geometry of the isoperiodic locus}
\label{chronos}

\subsection{Flat structure on the isoperiodic foliation}
The following `relative periods' coordinates on the isoperiodic leaf are well-known in Teichm\"uller theory \cite{GK}\cite{HM}.

\begin{fact}
Fix a cohomology class $[\alpha] \in H^1(S,\C)$, an isoperiodic family of complex structures w.~r.~t. $[\alpha]$, and let $\alpha_I \in \Omega(S,I)$ be a holomorphic 1-form representing it in the complex structure $I$. Let $z_0^I, z_1^I, \dots z_{2g-3}^I$ be its zeroes (so that the point $z_i^I$ varies smoothly as we vary the complex structure $I$). Then the tuple of functions $$\Xi_i(I) = \int_{z_0^I}^{z_i^I}\alpha_I$$ for $i = 1, 2,\dots 2g-3$ gives a local coordinate system on the isoperiodic leaf of the class $[\alpha]$.
\end{fact}

Note that each function $\Xi_i$ is defined up to an additive constant, i.~e. a period of $\alpha$.

Let us give a more algebraic handling of the above construction. Much like we considered the normal sheaf of a plane in $H^0(K)$, one can consider a normal sheaf for just one form $\alpha$. It is isomorphic to the structure sheaf $\cO_Z$ of the subscheme $Z = Z_\alpha$ of zeroes of $\alpha$, and the long exact sequence reads:
$$0 \to H^0(\cO) \to H^0(\cO_Z) \to H^1(T) \to H^1(\cO) \to 0.$$

Under the Serre duality the last arrow can be viewed as the arrow $H^0(K_S^2)^* \to H^0(K_S)^*$ dual to the arrow $H^0(K_S) \to H^0(K_S^2)$ given by $\xi \mapsto \xi\ox\alpha$. Its kernel is hence precisely the tangent space to the universal isoperiodic deformation w.~r.~t. $\alpha$.

\begin{expl}
Let $E$ be an elliptic curve with a holomorphic form $dz$, $p \colon S \to E$ a ramified cover, and $\alpha = p^*dz$. Then $Z$ is the ramification divisor of $p$, and the isoperiodic deformation w.~r.~t. $\alpha$ is given by the variation of the ramification points on $E$. Variation of each point is prescribed by a vector, which can be paired with $dz$ to obtain a number. Hence the variation is described by a tuple of numbers at each ramification point. The trivial deformations are precisely the ones coming from a variation of the ramification locus, which arises from a shift on the elliptic curve, i.~e. for which the numbers associated to the ramification points are all the same (so that the section of $H^0(\cO_Z)$ is constant).
\end{expl}

\begin{pr}
Let $v \in H^1(T)$ be an isoperiodic deformation w.~r.~t. $\alpha$, which we view both as a variation of complex structure $I \mapsto I+\eps v$ and a section $s \in H^0\left(\cO_{Z_\alpha}\right)$. Then one has $\Lie_v\Xi_i(I) = s(z_i) - s(z_0)$.
\end{pr}

Whenever we have a section $s \in H^0(\cO_Z)$, we shall denote the corresponding deformation by $v_s \in H^1(T)$.

\begin{pr}
Let $s \in H^0(\cO_{Z_\alpha})$ be a section. The value of its image under the connecting homomorphism in $H^1(T) \cong H^0(K^{\ox 2})^*$ on a quadratic differential $\omega$ is given by \begin{equation}\label{connect}v_s(\omega) = \sum_{i=0}^{2g-3}s(z_i){\Res_{z_i}}{\left(\frac{\omega}{\alpha}\right)}.\end{equation}
\end{pr}
\begin{proof}
Computation for \v{C}ech--Dolbeault double resolution \cite{Z}.
\end{proof}

It is clear that $v_s$ vanishes on each quadratic differential of the form $\alpha\ox-$, since the quotient $\omega/\alpha$ is holomorphic and has zero residues in this case, so that the corresponding deformations indeed preserve the periods of $\alpha$. It is also clear that the right-hand side vanishes for $v$ a constant vector, since in this case it is the sum of residues of a meromorphic 1-form, which is zero on a compact curve.

\subsection{Lesser isoperiodic leaf in a greater isoperiodic leaf}

The lesser isoperiodic leaf of $(\alpha, \beta)$ is obviously contained in the tangent space to the isoperiodic leaf of $\beta$, and one can describe the map on the tangent spaces $H^0(K)/\langle\alpha,\beta\rangle \to H^0(\cO_{Z_\alpha})/\mathrm{const}$.

Again, we examine thoroughly the case of a curve on an abelian surface. The forms $\alpha$ and $\beta$ give two left-invariant foliations on the surface, and the zeroes of $\alpha$ and $\beta$ correspond to the points where the curve is tangent to the foliation. Hence, interpreting 1-form as a displacement of the curve, we should look at its values in the zeroes of $\alpha$. Note that $\beta$ is never zero in a zero of $\alpha$, hence the mapping $H^0(K) \to H^0\left(\cO_{Z_\alpha}\right)$ given by \begin{equation}\label{gammabeta}\gamma \mapsto s^\beta_\gamma = \left(\frac{\gamma(z_0)}{\beta(z_0)}, \frac{\gamma(z_1)}{\beta(z_1)}, \frac{\gamma(z_2)}{\beta(z_2)}, \dots, \frac{\gamma(z_{2g-3})}{\beta(z_{2g-3})}\right)\end{equation} is well-defined. Moreover, after quotienting out the constants, the mapping vanishes on $\alpha$ and $\beta$, giving the desired map.

The same formula works for any pair of 1-forms $\alpha$ and $\beta$ with disjoint zeroes, not necessarily with discrete period lattice. Moreover, if zeroes are not disjoint, we can write down the map $H^0(K(-Y)) \to H^0\left(\cO_{Z_\alpha}\right)$ by the same formula since any section of $K(-Y)$ can be represented by a 1-form vanishing at each zero of $\alpha$ where $\beta$ also vanishes.

The formula (\ref{connect}) for $v_s(\omega)$ yields $$v_\gamma(\omega) = \sum_{i=0}^{2g-3}\frac{\gamma(z_i)}{\beta(z_i)}{\Res_{z_i}}{\left(\frac{\omega}{\alpha}\right)},$$ and for $\omega = \beta\ox-$ by Cauchy's index theorem one has $$v_\gamma(\omega) = 2\pi i\sum_{i=0}^{2g-3}\frac{\gamma(z_i)}{\beta(z_i)}\int_{\partial\Delta_i}\frac{\beta\ox-}{\alpha} = 2\pi i\sum_{i=0}^{2g-3}\int_{\partial\Delta_i}\frac{\gamma\ox-}{\alpha} = 0$$ since the functions $\gamma/\beta$ are holomorphic in the small neighborhoods of $z_i$ (which we here denote by $\Delta_i$). Hence these deformations indeed preserve the periods of $\beta$.

Let $\iota$ be a hyperelliptic involution on a curve $S$ again. We know that the zero locus of a holomorphic 1-form is preserved by $\iota$, so let us order the zeroes $z_i \in Z = Z_\alpha$ in such a way that $\iota(z_{2i}) = z_{2i+1}$. In what follows, we assume that no differential we consider vanishes at a fixed point of the hyperelliptic involution.

\begin{pr}\label{couplets}
Under such an ordering, the tangent space of the lesser isoperiodic leaf for $\alpha$ and $\beta$, provided they have no zeroes in common, is the space of sections $s \in H^0(\cO_Z)$ with $s(z_{2i}) = s(z_{2i+1})$ modulo constants. In particular, it does not depend on $\beta$.
\end{pr}
\begin{proof}
Any isoperiodic deformation w.~r.~t. $\alpha$ preserving also periods of $\beta$ is given by a section of the form (\ref{gammabeta}). Since the involution $\iota$ is hyperelliptic, for any two forms $\beta, \gamma \in \Omega(S)$ one has $$s^\beta_\gamma(z_{2i}) = \frac{\gamma(z_{2i})}{\beta(z_{2i})} = \frac{-\gamma(z_{2i})}{-\beta(z_{2i})} = \frac{(\iota^*\gamma)(z_{2i})}{(\iota^*\beta)(z_{2i})} = \frac{\gamma(\iota(z_{2i}))}{\beta(\iota(z_{2i}))} = \frac{\gamma(z_{2i+1})}{\beta(z_{2i+1})} = s^\beta_\gamma(z_{2i+1}).$$ The space of sections with $s(z_{2i}) = s(z_{2i+1})$ has dimension $g-1$, and its image in $H^1(T)$ has dimension $g-2$. By Corollary \ref{one-to-one}, it is the tangent space to the lesser isoperiodic leaf. 
\end{proof}

This implies that if a first-order deformations of a hyperelliptic curve preserves periods of two holomorphic 1-forms without common zeroes, it preserves periods of any holomorphic 1-form. Of course there cannot exist such an analytic deformation, since it would contradict the Torelli theorem. This really means that the lesser isoperiodic leafs in the Teichm\"uller space have maximal possible tangency along the locus of hyperelliptic curves.

This might look a little weird: it implies e.~g. that the span of $L_\alpha\cap L_\beta$ for any coprime $\alpha$ and $\beta$ contains all the monomials of the form $\omega\ox\omega'$, i.~e. the multiplication map $H^0(K_S)^2 \to H^0(K_S^2)$ is not surjective. One ought be not afraid of this: note that $\iota^*(\omega) = -\omega$, so $\iota$ acts on the range of the multiplication map $H^0(K_S)^2 \to H^0(K_S^2)$ as identity. However, there are always quadratic differentials which are acted upon by $\iota$ as $-\Id$. The hyperelliptic curves of genus greater than two are the only counterexamples to the infinitesimal Torelli theorem for curves, this is due to Max Noether \cite[Ch.~VIII]{TrAG}.

\begin{pr}
Let $v\in H^1(T)$ be a first-order deformation. It preserves the hyperelliptic involution iff $\iota^*v = v$.
\end{pr}
\begin{proof}
Clear on the level of cocycles.
\end{proof}

Note that if $\iota^*v = -v$, then one has $v(\omega\ox\omega') = -(\iota^*v)(\omega\ox\omega') = -v(\iota^*\omega\ox\iota^*\omega') = -v((-\omega)\ox(-\omega')) = -v(\omega\ox\omega')$. In other words, the $(-1)$-eigenspace of $\iota^*$ on $H^1(T)$ is the space of first-order deformations which are isoperiodic for any class. This can be seen e.~g. from the fact that the connecting homomorphism $H^0(K) \to H^1(T)$ is equivariant w.~r.~t. the hyperelliptic involution.

\begin{pr}\label{hyperel-skew}
Let $s \in H^0(\cO_Z)$ be a section which gives a deformation $v_s \in H^1(T)$. Then it satisfies $s(z_{2i}) = -s(z_{2i+1})$, maybe after adding a constant section. In particular, a hyperelliptic deformation of a hyperelliptic curve never preserves periods of more than one holomorphic 1-form.
\end{pr}

\begin{expl}
Since this is not really important for us, we shall consider an example instead of giving a proof. Let $E$ be an elliptic curve, and $D$ is an effective divisor symmetric w.~r.~t. $z\mapsto-z$. Let $S$ be a ramified cover of $E$ with ramification divisor $D$. Then $S$ admits a hyperelliptic involution s.~t. the diagram

$$\begin{CD}
S @>>> \C\P^1\\
@V{p}VV @VV{2:1}V\\
E @>>> \C\P^1
\end{CD}$$

is commutative. Let $\alpha = p^*dz$, the isoperiodic deformations for $\alpha$ are given by the displacements of points in the ramification divisor $D$. If we want the hyperelliptic involution to survive, we should move the points $\zeta \in D$ and $-\zeta \in D$ in opposite directions, so that their images would be interchanged by $z \mapsto -z$ as well.
\end{expl}

\begin{expl}
This last assertion can also be demonstrated in a more specific situation. Let $S$ be a hyperelliptic curve on an abelian surface $A$, and $\alpha,\beta$ be restrictions of the forms on $A$ to $S$. The hyperelliptic involution on $S$ extends to $A$, and $S$ projects to a rational curve on the K3 surface $\widetilde{A/\iota}$. An isoperiodic deformation would produce a variation of this hyperelliptic curve inside $A$ and hence give rise to a family of rational curves on a K3 surface. However, rational curves on a K3 surface in a given homology class are discrete. This is a rephrasing of Pirola's theorem on rigidity of hyperelliptic curves on abelian varieties \cite[Remark 1]{P} (proved over arbitrary fields by Oort and de Jong \cite{OdJ}).
\end{expl}

The Proposition \ref{hyperel-skew} can be furnished in a manner similar to our main Proposition \ref{three-diff}:

\begin{pr}\label{hype-two-diff}
Let $(S, I)$ be a hyperelliptic complex curve of genus at least three, and $\tau \subset H^{1,0}(S,I)$ be a generic two-dimensional subspace. There exist a neighborhood $U \subset \Gris(2,2g)$ containing $\tau$ s.~t. for any $\tau' \in U$ there exists a unique hyperelliptic complex structure $I' = I(\tau')$ s.~t. one has $\tau' \subset H^{1,0}(S,I')$. In other words, deformation of a generic pair of abelian differentials determines a unique local deformation of a hyperelliptic complex structure, so that the deformed cohomology classes would be represented by abelian differentials in the deformed complex structure.
\end{pr}
\begin{proof}
The mapping class group preserves the hyperelliptic locus. Otherwise the proof is parallel to the proof of the Proposition \ref{three-diff}.
\end{proof}

\begin{expl}
Note that the assumption on nonvanishing of 1-forms in the fixed points of $\iota$ is necessary. The counterexample can be given by a double cover of genus two curve $C$ ramified in a divisor, which is preserved by the hyperelliptic involution of $C$ (it is clear from a commutative diagram like above that such a cover is also hyperelliptic). The pullbacks of two forms on $C$ have their periods preserved while we vary the ramification divisor (provided it stays symmetric w.~r.~t. the involution). Note that in this case the pullbacks vanish in the preimages of the ramification divisor, which are preserved by the hyperelliptic involution of the double cover. Also note that such double covers have genus at least four, so their existence does not contradict the following
\end{expl}

\begin{pr}
Any pair of holomorphic linked classes on a genus three curve admits an isoperiodic deformation in which they are coprime---that is, the coprime Kapovich--Schottky locus exhausts the Kapovich--Schottky locus in genus three.
\end{pr}
\begin{proof}
By Corollary \ref{gen-three}, genus three curve with linked classes is hyperelliptic. By the above Proposition, any first-order deformation of a hyperelliptic curve preserving periods of two differentials preserves period of the third. On the other hand, any pair of abelian differentials on genus three curve possesses a one-parameter analytic family of isoperiodic deformation. If it lies within the hyperelliptic locus, it preserves the periods of all the three differentials, and hence is trivial by Torelli theorem. Therefore it intersects the complement of the hyperelliptic locus, which gives the desired deformation.
\end{proof}

Note that this Proposition cannot be pushed any further since there exist ramified covers of genus two curves with deformations of dimension larger than predicted. The hyperelliptic curves are also the instance of triples of classes in which the derivative of the polyperiod mapping is degenerate, but its fibers are still of predicted dimension.

Let us also mention a following result.

\begin{pr}
For a generic pair $(I,\tau) \in \Gr(2,\Omega\T_S)$, where $I \in \T_S$, $\tau \subset \Omega(S,I)$ there exist a hyperelliptic complex structure $J \in \T_S$ s.~t. one has $\tau \subset \Omega(S,J)$.
\end{pr}
\begin{proof}
The set of bidifferentials which can be realized in a hyperelliptic complex structure is open and invariant under an ergodic action.
\end{proof}

Applying this Proposition to a pair of classes coming from an abelian surface, one obtains a nice

\begin{corol}\label{def-to-hype}
Let $\{C \subset A^2\}$ be a generic pair of an abelian surface $A^2$ and a curve $C$ on it with at worst normal crossings. Then there exists a hyperelliptic curve within $A^2$ in the same homology class $[C]$.
\end{corol}

This theorem can be thought of as a sibling of the Bogomolov--Mumford theorem on existence of rational curves on a K3 surface, especially if one proves that this hyperelliptic curve may be achieved by an actual deformation (which would follow from the connectedness of the lesser isoperiodic leaves).

\subsection{Sections of determinantal varieties and reciprocity law}\label{recip}
It is not as simple to describe the range of the map $\gamma \mapsto s_\gamma$ for non-hyperelliptic curves. Something can be said, though.

Let $S$ be a genus $g$ curve, and $\tau \subset \Omega(S)$ be a subspace. Let us denote by $H^1_\tau(T)$ the range of the connecting homomorphism $H^0(\nu_\tau) \to H^1(T_S)$, which is the space of first-order deformations preserving the periods of forms from $\tau$.

\begin{defn}
Let $\alpha \in \Omega(S)$ be a holomorphic differential. The set $$Q(S,\alpha) = \bigcup_{\alpha\in\tau}H^1_\tau(T) \subset H^1_{\alpha}(T)$$ is called the {\bf doubly isoperiodic cone}.
\end{defn}

\begin{pr}\label{hypers}
The doubly isoperiodic cone for a generic pair $(S,\alpha)$ has dimension $2g-4$ (and hence codimension one in $H^1_\alpha(T_S)$).
\end{pr}
\begin{proof}
Note that if $H^1_\tau$ and $H^1_{\tau'}$ have nonzero intersection, it means that $H^1_{\tau,\tau'}$ is nonzero. That is, the self-intersection of this cone correspond to triples of forms possessing nontrivial infinitesimal deformations preserving periods. Since generic triple has no such deformations, this means that an open subset of the cone is fibered with fibers $H^1_\tau$ (of dimension $g-2$) over the base parametrizing different choices for $\tau$ (i.~e. $\P(\Omega/\langle\alpha\rangle)$, which has dimension $g-2$ as well). 
\end{proof}

\begin{defn}
Let $V$ be a vector space, and $v \in V$ be a nonzero vector. The cone of elements in $\Hom(V/\langle v\rangle,V^*)$ which have nonzero kernel is called the {\bf linear-algebraic doubly isoperiodic cone} and denoted by $\mathfrak{Q}(V,v)$.
\end{defn}

\begin{pr}\label{legen}
Obviously, the doubly isoperiodic cone $Q(S,\alpha)$ is the intersection of the cone $\mathfrak{Q}(\Omega(S),\alpha)$ and the subspace $H^1_\alpha(T_S)$ embedded via the Kodaira--Spencer maps. This intersection is unlikely.
\end{pr}
\begin{proof}
The ambient space $\Hom(\Omega/\langle\alpha\rangle,\Omega^*)$ has dimension $(g-1)g = g^2-g$. In order to give an element in the cone $\mathfrak{Q}(\Omega,\alpha)$, one has to specify the kernel (there is $(g-2)$-dimensional space of possibilities for it), and the map to $\Omega^*$ (which can be chosen out of $(g-2)g$-dimensional variety). Hence the dimension of the linear-algebraic doubly isoperiodic cone is $(g-2)(g+1) = g^2 - g - 2$, and its codimension equals two. However, the codimension of the doubly isoperiodic cone $Q(S,\alpha) \subset H^1_\alpha(T)$ is one. Therefore the intersection of the linear-algebraic doubly isoperiodic cone with the subspace $H^1_\alpha(T) \subset \Hom(\Omega/\langle\alpha\rangle,\Omega^*)$ consisting of the Kodaira--Spencer operators of the deformations preserving the periods of $\alpha$, is unlikely.
\end{proof}

This Proposition gives a condition on the possible position of $H^1_\alpha$ subspace within the tangent space to the moduli of abelian varieties. It can be probably easily described by the means of representation theory, since the cone $\mathfrak{Q}(V,v)$ is a union of orbits of a parabolic subgroup in $\mathrm{GL}(V)$ stabilizing the line spanned by $v$. Since the tangent space to the moduli space is spanned by different $H^1_\alpha$ subspaces, this statement may be regarded as a higher-order Legendrian property for the Schottky locus.

The linear-algebraic doubly isoperiodic cone $\mathfrak{Q}(V,v)$ is determined by vanishing of a combination of $(g-1)\x(g-1)$ minors, so it has degree $g-1$---and so has the doubly isoperiodic cone $Q(S,\alpha)$, being its linear section. The projectivization $\P Q(S,\alpha)$ is a hypersurface in $\P(H^1_\alpha) = \P^{2g-4}$ swept by a family of $(g-3)$-dimensional subspaces $\P(H^1_\tau)$ parametrized by $\P^{g-2} = \P(\Omega/\langle\alpha\rangle)$. The first genus when it nontrivial is $g=3$, when it is a plane quadric.

\begin{pr}
Let $X$ be a genus three curve, and a form $\alpha \in \Omega(X)$ has zeroes $z_0, z_1, z_2, z_3 \in X$. There exists a number $B$ s.~t. no matter which 1-forms $\beta,\gamma\in \Omega(X)$ we take (provided they have no zeroes in common and, together with $\alpha$, span $\Omega(X)$), one has $$\left[\frac{\gamma(z_0)}{\beta(z_0)}~\colon \frac{\gamma(z_1)}{\beta(z_1)}~\colon \frac{\gamma(z_2)}{\beta(z_2)}~\colon \frac{\gamma(z_3)}{\beta(z_3)}\right] = B,$$ where the braces denote the cross-ratio. Moreover, any quadruple of numbers with cross-ratio $B$ can be realized as the section $s^\beta_\gamma$ for some $\beta,\gamma\in\Omega(X)$.
\end{pr}
\begin{proof}
Any pair $\beta',\gamma' \in \Omega(X)$ can be obtained from the pair $\beta,\gamma$ by the operations $(\beta,\gamma) \mapsto (\beta+c\alpha,\gamma+c\alpha)$, $(\beta,\gamma)\mapsto(c\beta,c\gamma)$, $(\beta,\gamma)\mapsto(\beta,\gamma+c\beta)$, $(\beta,\gamma)\mapsto(\gamma,\beta)$. The first operation does not change the vector $s^\beta_\gamma$ at all, since $\alpha$ vanishes at $z_i$'s. The second operation scales the quadruple and hence does not change its cross-ratio. The third operation adds a constant quadruple $(c,c,c,c)$, so does not affect the cross-ratio as well. The last operarion acts on the quadruple as $(a,b,c,d)\mapsto\left(a^{-1},b^{-1},c^{-1},d^{-1}\right)$. It is elementary to check that $$\frac{a^{-1}-c^{-1}}{b^{-1}-c^{-1}}:\frac{a^{-1}-d^{-1}}{b^{-1}-d^{-1}} = \frac{bc}{ac}\cdot\frac{c-a}{c-b}:\left(\frac{bd}{ad}\cdot\frac{d-a}{d-b}\right) = \frac{a-c}{b-c}:\frac{a-d}{b-d}.$$

On the other hand, any two quadruples with the same cross-ratio can be related by the chain of operations $(a,b,c,d)\mapsto(ka,kb,kc,kd)$, $(a,b,c,d)\mapsto(a+k,b+k,c+k,d+k)$, $(a,b,c,d)\mapsto(a^{-1},b^{-1},c^{-1},d^{-1})$.
\end{proof}

This cross-ratio can be also captured by the means of elementary geometry.

\begin{pr}
Let $X$ be a smooth plane quartic (i.~e. a canonical embedding of non-hyperelliptic curve of genus three), and $\alpha$ a holomorphic 1-form (so that its zeroes are cut out by a straight line). Then the cross-ratio $B(X,\alpha)$ defined above equals to the cross-ratio of these zeroes considered as four points on a projective line (up to relabeling the indices).
\end{pr}
\begin{proof}
Points in $\P(\Omega(X)^*)$ correspond to (hyper)planes in $\Omega(X)$. Here, a point $z_i$ corresponds to a plane of holomorphic 1-forms vanishing at $z_i$. If one picks the basis $\alpha,\beta,\gamma$ s.~t. $\beta(z_0) = 0$, $(\beta-\gamma)(z_1) = 0$, $\gamma(z_2) = 0$, then the ratio $\beta(z_3)/\gamma(z_3)$ is determined uniquely, and on the one hand equals to the cross-ratio of four planes spanned by $\alpha$ and another 1-form vanishing at $z_i$ (which equals the cross-ratio of four points on a line cutting out the zeroes of $\alpha$), and on the other hand by construction equals to $B(X,\alpha)$.
\end{proof}

Take $B \neq \pm 1$. Then the set of quadruples $(a,b,c,d) \in \C^4$ with $[a:b:c:d]=B$, after quotienting out the line spanned by $(1,1,1,1)$, is a quadratic cone. This endows the greater isoperiodic leaf for $(X,\alpha)$, which is three-dimensional, with a meromorphic field of quadratic cones, which is degenerate e.~g. along the hyperelliptic locus. By the above Proposition, any lesser isoperiodic leaf, which in this case has dimension one, is tangent to this field of cones, i.~e. isotropic w.~r.~t. the holomorphic local conformal structure.

This can be thought of as a Weil-type reciprocity law. However, whereas the usual reciprocity law involves two functions and values of each of them at zeroes and poles of the other, this reciprocity law involves values of two holomorphic forms at zeroes of the third, so in order to deduce from the Weil reciprocity law itself, some terms must be canceled out. For greater genera, a more difficult reciprocity relation of degree $g-1$ on the values of the function $\gamma/\beta$ at zeroes of a third form $\alpha$ emerges, and we are unable to write it down explicitly for any pair $(S,\alpha)$. It might be an interesting task.

\begin{expl}
Let $S$ be a genus $g=4$ curve. In this case, the projectivization $\P Q(S,\alpha)$ is a hypersurface swept by lines parametrized by $\P^2$. It is cut out, on the other hand, as a plane section of a determinantal variety $\P\mathfrak{Q}(\Omega(S),\alpha)$, which has degree three. Hence $\P Q(S,\alpha)$ is a cubic threefold. It is known that the so-called Fano surface of lines on a smooth cubic threefold is of general type, and hence cannot be $\P^2$. Therefore the doubly isoperiodic cone for a genus four curve is never smooth. This implies the following

\begin{pr}
For any homolorphic 1-form $\alpha$ on a genus four curve $S$ there exist a linked three-dimensional plane $\tau \subset \Omega(S)$ containing $\alpha$.
\end{pr}
\begin{proof}
Indeed, the singularities of the doubly isoperiodic cone correspond to linked triples (cf. the proof of Proposition \ref{hypers}).
\end{proof}

Note that there is another way of associating a singular cubic threefold to a non-hyperelliptic genus four curve, via the linear system of cubics passing through its canonical image \cite{vdGK}. It does not require a choice of a holomorphic 1-form, so it's probably not the same, though it would be interesting to find a direct relation between these threefolds.
\end{expl}

%\section{Case of common zeroes: possibilities and limitations}
%\label{common-zeroes}
%
%\section{Singular curves on abelian surfaces}
%\label{lace}

\paragraph*{Acknowledgements.}
I~am indebted to Fedor Alexeyevich Bogomolov, my thesis advisor, whose genius and insightfulness led me through the proofs of the propositions which I cannot help myself to expect were true. I~am grateful to Renat Abugaliev, Enrico Arbarello, Alexander Berdnikov, Samuel Grushevsky, Igor Krichever, Grigory Papayanov, Alexander Petrov, David Loeffler, Vasily Rogov, Lev Soukhanov, Misha Verbitsky, and Sasha Zakharov, whose remarks were fruitful and sometimes essential for the writing of the paper. I am also grateful to Nikon Kurnosov, who gave me a possibility to present a preliminary draft of this paper on his seminar, and to Dmitry Kaledin, who, being present at that seminar, discouraged me from an attempt to use complex analysis (so-called Schiffer variations) in order to repair the case of two differentials with common zeroes, and to examine it through the optics of Algebraic Geometry (which I did, only to find that the problem is much richer than I expected).

\noindent {\sc rodion n. d\'eev\\
Independent University of Moscow\\
119002, Bolshoy Vlasyevskiy Pereulok 11\\ 
Moscow, Russian Federation} \\
\tt deevrod@mccme.ru, also:\\
{\sc Department of Mathematics\\
Courant Institute, NYU \\
251 Mercer Street \\
New York, NY 10012, USA,} \\
\tt rodion@cims.nyu.edu

\end{document}